\documentclass[smallextended,referee,envcountsect]{svjour3}
\usepackage{booktabs}
\usepackage{amscd}
\smartqed
\usepackage{graphicx,amssymb}
\usepackage{enumerate,amsfonts}
\usepackage{amsmath}
\usepackage[bookmarksnumbered, plainpages, colorlinks]{hyperref}
\usepackage{multicol,multirow}
\usepackage{subfigure}
\usepackage{threeparttable}
\usepackage[top=2cm, bottom=2cm, left=2cm, right=2cm]{geometry}
\usepackage{float}
\usepackage[section]{placeins}
\usepackage{caption}
\usepackage[ruled,vlined]{algorithm2e} 
\usepackage{bigstrut}

\numberwithin{equation}{section}

\begin{document}
\baselineskip=16pt

\title{Self-adaptive-type CQ algorithms for split equality problems}

\titlerunning{Self-adaptive-type CQ algorithms}

\author{Songxiao Li$^{1}$ \and Bing Tan$ ^1 $ \and Zheng Zhou$ ^{1,*}$}

\authorrunning{S. Li, B. Tan, Z. Zhou}

\institute{$^*$Corresponding author.\at
Songxiao Li \at
\email{jyulsx@163.com}\\
Bing Tan \at
\email{bingtan72@gmail.com}\\
 Zheng Zhou \at
\email{zhouzheng2272@163.com}
\and
$^1$ Institute of Fundamental and Frontier Sciences, University of Electronic Science and Technology of China, Chengdu 611731, China}

\date{Received: date / Accepted: date}

\maketitle

\begin{abstract}
The purpose of this paper is concerned with the approximate solution of split equality problems. We introduce two types of algorithms and a new self-adaptive stepsize without prior knowledge of operator norms. The corresponding strong convergence theorems are obtained under mild conditions. Finally, some numerical experiments demonstrate the efficiency of our results and compare them with the existing results.

\end{abstract}

\section{Introduction}

As an extension of the split feasibility problem (for short, SFP), in 2013, Moudafi \cite{moudafi2013relaxed,moudafi2014alternating} introduced the following split equality problem (for short, SEP) that is applied to decomposition for PDEs \cite{attouch2008alternating} and intensity-modulated radiation therapy \cite{censor2006unified}. Let $H_1$, $H_2$ and $H_3$ be Hilbert spaces and $C\subset H_1$, $Q\subset H_2$ be nonempty closed convex subsets. Let $A:H_1\rightarrow H_3$ and $B:H_2\rightarrow H_3$ be bounded linear operators. The split equality problem is to find
\begin{equation}
\label{sep}
x^*\in C, y^*\in Q \ {\text{ such\ that}}\ \ Ax^*=By^*.
\end{equation}
In particular, when $B=I$ and $H_2=H_3$, the SEP can be considered as well-known the split feasibility problem (Censor and Elfving introduced in \cite{s2}), which is to find $x^*\in C$ such that $Ax^*\in Q$. Naturally, $x^*$ is a solution of the split feasibility problem if and only if $x^*$ is a solution of the equation $x^*=P_C(I-\gamma A^*(I-P_Q)A)x^*$, where $P_C:H_1\rightarrow C$ and $P_Q:H_2\rightarrow Q$ are metric projection operators, $A^*$ is the adjoint operator of $A$. By virtue of the fixed point algorithm, Byrne \cite{s1} came up with CQ algorithm to approximate a solution of the split feasibility problem by the recursive procedure $x_{n+1}=P_C(I-\gamma A^\mathrm{T}(I-P_Q)A)x_n$, where $A^\mathrm{T}$ is the matrix transposition of $A$ and $\gamma\in (0,2/L)$ with $L$ the largest eigenvalue of matrix $A^\mathrm{T}A$. Subsequently, Wang \cite{wang2017new}, Yao, Liou and Postolache \cite{yao2018self} studied the following new iterative algorithm
\begin{equation}
\label{nsep}
x_{n+1}=x_n-\gamma_n[(I-P_C)+ A^*(I-P_Q)A)]x_n,\ \forall n\geq 0,
\end{equation}
where $\{\gamma_n\}$ is a self-adaptive stepsize sequence without prior knowledge of
operator norms. It is worth noting that only the weak convergence of split feasibility problem were obtained by Byrne's algorithm and the algorithm \eqref{nsep}. To fill this gap, Halpern algorithms and viscosity algorithms were studied by many authors and were applied to various aspects, such as fixed point problems \cite{qin2009convergence,qin2008strong,syoaml2011,halpern1967fixed,luo2020,tan2020strong}, variational inequality problems \cite{ansariaa2020,Chang209for,choacta2012,chojia2013,kraikaew2014strong,takahahsifpt2018}, equilibrium problems \cite{takahashi2007viscosity}, split feasibility problems \cite{sahuo,xu2006variable}, and so on. Based on the idea of Halpern algorithm (Halpern introduced in \cite{halpern1967fixed}), Xu \cite{xu2006variable} proposed the following modified algorithm and obtained strong convergence of the split feasibility problem
\begin{equation}
\label{Hsep}
x_{n+1}=\alpha_nu+(1-\alpha_n)P_C(I-\gamma A^*(I-P_Q)A)x_n,
\end{equation}
where $\gamma$ is a constant in $(0, {2}/{\|A\|^2})$ and $u$ is a fixed point. Further, Takahashi \cite{takahashi2017mann,takahashi2018weak} proposed a modified Halpern algorithm, which uses a sequence $\{u_n\}$ converges strongly to $u$, and obtained the corresponding strong convergence theorems. The viscosity algorithm was introduced by Moudafi \cite{moudafi2000viscosity} in 2000, which uses a contraction mapping to approximate a solution of the fixed point problem.

On the other hand, to approximate the solutions of SEP, Moudafi \cite{moudafi2014alternating} presented the following alternating CQ algorithm (for short, ACQA)
\begin{equation}\label{ACQA}
\left\{
\begin{aligned}
&{x_{n+1}=P_C(x_n-\gamma_nA^*(Ax_n-By_n))},\\
&{y_{n+1}=P_Q(y_n+\gamma_nB^*(Ax_{n+1}-By_n)),}
\end{aligned}
\right.
\end{equation}
where $\{\gamma_n\}$ is a sequence in $(\varepsilon, \min\{\frac{1}{\|A\|^2}, \frac{1}{\|B\|^2}\}-\varepsilon)$ ($\varepsilon$ is a small enough nonnegative real number). Further, Byrne and Moudafi \cite{byrne2012extensions} came up with the following simultaneous CQ algorithm (for short, SCQA) to solve SEP, for $\varepsilon< \gamma_n < \frac{2}{\|A\|^2+\|B\|^2}-\varepsilon$,
\begin{equation}\label{SCQA}
\left\{
\begin{aligned}
&{x_{n+1}=P_C(x_n-\gamma_nA^*(Ax_n-By_n))},\\
&{y_{n+1}=P_Q(y_n+\gamma_nB^*(Ax_{n}-By_n)).}
\end{aligned}
\right.
\end{equation}

Due to expensive calculation of the projection operators $P_C$ and $P_Q$, Moudafi considered level sets to solve SEP, in which the level set of convex function is easy to implement, that is, $C$ and $Q$ are replaced with level sets of convex and subdifferentiable functions $\mathfrak{f}:H_1\rightarrow R$ and $\mathfrak{g}:H_2\rightarrow R$, respectively, i.e.,
$C=\{x\in H_1: \mathfrak{f}(x)\leq 0\}$, and $ Q=\{y\in H_2: \mathfrak{g}(y)\leq0\}$.
In this situation, Moudafi \cite{moudafi2013relaxed} put forward the relaxed alternating CQ algorithm (for short, RACQA)
\[
\left\{
\begin{aligned}
&{x_{n+1}=P_{C_n}(x_n-\gamma A^*(Ax_n-By_n))},\\
&{y_{n+1}=P_{Q_n}(y_n+\gamma B^*(Ax_{n+1}-By_n)),}
\end{aligned}
\right.
\]
where $\gamma$ is a constant in $(0, \min\{\frac{1}{\|A\|^2}, \frac{1}{\|B\|^2}\})$,
$C_n=\{x\in H_1|\mathfrak{f}(x_n)+\langle\xi_n,x-x_n\rangle\leq0\},\ \xi_n\in\partial \mathfrak{f}(x_n)$,
and $Q_n=\{y\in H_2|\mathfrak{g}(y_n)+\langle\eta_n,y-y_n\rangle\leq0\},\ \eta_n\in\partial \mathfrak{g}(y_n)$.
Especially, ACQA, SCQA and RACQA only got weak convergence properties of the split equality problem. Inspiration and motivation through these work in Moudafi \cite{moudafi2013relaxed,moudafi2014alternating,moudafi2000viscosity} and Takahashi \cite{takahashi2017mann,takahashi2018weak}, we introduce two modified self-adaptive-type iterative algorithms to solve the split equality problem \eqref{sep} in infinite Hilbert spaces by the Halpern algorithm and the viscosity algorithm. The corresponding strong convergence theorems are obtained without prior knowledge of operator norms. Furthermore, some numerical experiments are used to demonstrate and show the efficiency of our main results.

The present is built up as follows. Some basic properties and relevant lemmas will be introduced in Section \ref{sec2}, which will be used in the proof for the convergence of the proposed algorithms. The main results and some corollaries of this paper are contained in Sections \ref{se3} and \ref{se4}. The last section, in Section \ref{se5}, some numerical experiments demonstrate the efficiency of our results and compare them with the existing alternating CQ algorithm (ACQA) in Moudafi \cite{moudafi2014alternating}, simultaneous CQ algorithm (SCQA) in Byrne and Moudafi \cite{byrne2012extensions}, Dong, He and Zhao algorithm in \cite{dong2015solving}.

\section{Preliminaries}\label{sec2}

For the convenience and standard in the rest of this article, we use the notations $\rightarrow$ and $\rightharpoonup$ to represent strong convergence and weak convergence, respectively. The fixed point set of the mapping $T$ is represented by $F(T)$. Some well-known basic properties are as follows:
\begin{description}
	\item[\textbf{(P1)}] $P_C$ is denoted metric projection from $H$ onto $C$, that is, $P_Cx=\operatorname{argmin}_{y\in C} \|x-y\|,\ \forall x\in H$. It has such an equivalent form $ \langle P_Cx-x, P_Cx-y\rangle \le 0, \forall y \in C$, and can also be converted to $\|y-P_{C} x\|^{2}+\|x-P_{C} x\|^{2} \leq\|x-y\|^{2}$;
    \item[\textbf{(P2)}] The mapping $T:H\rightarrow H$ with $F(T)\neq\emptyset$ and $I-T$ is demiclosed at 0, i.e., for any sequence $\{x_n\}\in H$, if $\{x_n\}$ weakly converges to $x$ and $(I-T)x_n$ strongly converges to $0$, then $x\in F(T)$;

    \item[\textbf{(P3)}] The mapping $T:H\rightarrow H$ is a contraction with constant $\lambda$, that is,
    \[
    \|f(x)-f(y)\|\leq \lambda\|x-y\|,\ \forall x,y\in H,\ \lambda\in [0,1);
    \]
    \item[\textbf{(P4)}] $\partial \mathfrak{f}$ is denoted the subdifferential of convex function $\mathfrak{f}:H\rightarrow R$ at $x$, that is,
\[
\partial \mathfrak{f}(x)=\{\varpi: \mathfrak{f}(y)\ge \mathfrak{f}(x)+\langle \varpi,y-x\rangle,\ \forall y\in H\}.
\]
    \item[\textbf{(P5)}] For any $x,y\in H$, the following properties hold
    \[
    \|x+y\|^2\leq \|x\|^2+2\langle y, x+y\rangle;
    \]
    \[
    \|\kappa x+(1-\kappa)y\|^2=\kappa\|x\|^2+(1-\kappa)\|y\|^2-\kappa(1-\kappa)\|x-y\|^2,\ \forall \kappa\in R.
    \]
\end{description}

\begin{lemma}\label{lem1}
Let the solution set of the split equality problem \eqref{sep} is nonempty. For any $\gamma>0$, a solution of SEP is equivalent to a solution of the following equations
\begin{equation}\label{eqs}
\left\{\begin{aligned}
&x=x-\gamma\left((I-P_C)x+A^*(Ax-By)\right),\\
&y=y-\gamma\left((I-P_Q)y-B^*(Ax-By)\right).\end{aligned}\right.
\end{equation}
\end{lemma}

\begin{proof}
Obviously, any solution of the problem \eqref{sep} is the solution of equations \eqref{eqs}. On the other hand, put any element $(x,y)$ in the solution set of equations \eqref{eqs}, we have
\[
\left\{\begin{aligned}
&0=(I-P_C)x+A^*(Ax-By),\\
&0=(I-P_Q)y-B^*(Ax-By).\end{aligned}\right.
\]
For any $(x^*,y^*)$ in the solution set of SEP, that is, $x^*\in C$, $y^*\in Q$ and $Ax^*=By^*$, we get
\[
\begin{aligned}
0&=\langle (I-P_C)x+A^*(Ax-By),x-x^*\rangle\\
&=\langle x-P_Cx,x-P_Cx\rangle+\langle x-P_Cx,P_Cx-x^*\rangle+\langle Ax-By,Ax-Ax^*\rangle\\
&\geq \|x-P_Cx\|^2+\langle Ax-By,Ax-Ax^*\rangle,
\end{aligned}
\]
and
\[
\begin{aligned}
0=\langle (I-P_Q)y-B^*(Ax-By),y-y^*\rangle\geq\|y-P_Qy\|^2-\langle Ax-By,By-By^*\rangle.
\end{aligned}
\]
Combine the above two formulas we get $0\geq\|x-P_Cx\|^2+\|y-P_Qy\|^2+\|Ax-By\|^2$. This implies that $x\in C$, $y\in Q$ and $Ax=By$, i.e., $(x,y)$ is a solution of the split equality problems \eqref{sep}.
\end{proof}

\begin{lemma} \emph{\cite{he2013solving}}\label{lem2.5}
Let $\{\theta_{n}\}$ and $\{\eta_n\}$ be two nonnegative real numbers sequences such that
\[
\theta_{n+1}\leq (1-\delta_{n})\theta_{n}+\delta_n\tau_n, n\geq 0,
\]
and
\[
\theta_{n+1}\leq \theta_{n}-\eta_n+\zeta_n, n\geq 0,
\]
where $\{\tau_n\}$, $\{\zeta_n\}$ and $\{\delta_{n}\}$ are real sequences with $0< \delta_{n} <1$. If
\begin{itemize}
	\item $\sum_{n=0}^\infty {\delta_{n}}= \infty$ and  $\underset{n\rightarrow \infty}{\lim}\zeta_n=0$;
	\item $\underset{k\rightarrow \infty}{\lim}\eta_{n_k}=0$ implies $\underset{k\rightarrow \infty}{\limsup}\ \tau_{n_k}\leq0$ where $\{n_k\}$ is any real numbers subsequence of $\{n\}$.
\end{itemize}
The sequence $\{\theta_n\}$ is convergent to $0$ as $n\rightarrow \infty$.
\end{lemma}

\section{Halpern-type CQ algorithms}\label{se3}

In this section, we propose two Halpern-type algorithms to approximate a solution of the split equality problem \eqref{sep}, and assume that the solution set of SEP \eqref{sep} is nonempty, i.e., $\Omega=\{(x^*,y^*)\in C\times Q| Ax^*=By^*\}\neq \emptyset$. In addition, the following assumptions are presupposed.
\begin{description}
	\item[\textbf{(A1)}] $H_1$, $H_2$ and $H_3$ are Hilbert spaces, and $C\subset H_1$, $Q\subset H_2$ are nonempty closed convex subsets;
	\item[\textbf{(A2)}] $A:H_1\rightarrow H_3$, $B:H_2\rightarrow H_3$ are bounded linear operators, $A^*$ and $B^*$ are the adjoint operators of $A$ and $B$, respectively;
    \item[\textbf{(A3)}] $\{u_n\}\subset H_1$, $\{v_n\}\subset H_2$ are two convergence sequences such that $u_n\rightarrow u$ and $v_n\rightarrow v$.
\end{description}

\subsection{Self-adaptive Halpern-type CQ algorithm (SHCQA)}
Through the aforementioned Halpern-type algorithm in \cite{takahashi2017mann,takahashi2018weak}, the iterative sequence $\{(x_n, y_n)\}$ of the split equality problem \eqref{sep} is generated by the following recursive procedure
\begin{equation}\label{CQ}
		\left\{\begin{aligned}
&{\widehat{x}_{n}=x_{n}-\gamma_n\left[(I-P_C)x_{n}+A^*(Ax_{n}-By_{n})\right],}\\
&{x_{n+1}=\delta_n u_n+(1-\delta_n)\widehat{x}_{n},}\\
&{\widehat{y}_{n}=y_{n}-\gamma_n\left[(I-P_Q)y_{n}-B^*(Ax_{n}-By_{n})\right],}\\
&{y_{n+1}=\delta_n v_n+(1-\delta_n)\widehat{y}_{n}, n\geq 0.}
\end{aligned}\right.
\end{equation}
\noindent The corresponding parameters satisfy the following restrictions
\begin{description}
    \item[\textbf{(R1)}] If $Ax_{n}\neq By_{n}$, the self-adaptive stepsize
\[
\gamma_n=\alpha_n\min\left\{1,\frac{\|Ax_n-By_n\|^2}{\|A^*(Ax_{n}-By_{n})\|^2+\|B^*(Ax_{n}-By_{n})\|^2}\right\}\ {\text{with}}\ \alpha_n\in(0,1).
\]
         Otherwise, the self-adaptive stepsize $\gamma_n=\alpha_n$;
    \item[\textbf{(R2)}] $\{\delta_n\}$ is real number sequence in $(0,1)$ such that $\lim_{n\rightarrow \infty}\delta_n=0$ and $\sum_{n=0}^{\infty}\delta_n=\infty$.
\end{description}

\begin{theorem}\label{sep1}
Given the assumptions {(A1)-(A3)} and conditions {(R1)-(R2)}. The iterative sequence $\{(x_n, y_n)\}$ generated by algorithm \eqref{CQ} converges strongly to $P_\Omega(u,v)\in \Omega$.
\end{theorem}

\begin{proof}
Take $(x^*, y^*)=P_\Omega(u,v)\in \Omega$, that is, $x^*\in C$, $y^*\in Q$ and $Ax^*=By^*$. Using algorithm \eqref{CQ} and the nature of projection operator $P_C$ to get
\begin{equation}\label{zn}
\begin{aligned}
\|\widehat{x}_{n}-x^*\|^2&=\|x_{n}-x^*\|^2-2\gamma_n\langle (I-P_C)x_{n}+A^*(Ax_{n}-By_{n}), x_{n}-x^*\rangle\\
&\quad+\gamma_n^2\|(I-P_C)x_{n}+A^*(Ax_{n}-By_{n})\|^2\\
&\leq \|x_{n}-x^*\|^2-2\gamma_n\|(I-P_C)x_{n}\|^2-2\gamma_n\langle Ax_{n}-By_{n}, Ax_{n}-Ax^*\rangle\\
&\quad+2\gamma_n^2\left(\|(I-P_C)x_{n}\|^2+\|A^*(Ax_{n}-By_{n})\|^2\right).
\end{aligned}
\end{equation}
Similarly, the following inequality is available
\begin{equation}\label{sn}
\begin{aligned}
\|\widehat{y}_{n}-y^*\|^2&\leq \|y_{n}-y^*\|^2-2\gamma_n\|(I-P_Q)y_{n}\|^2+2\gamma_n\langle Ax_{n}-By_{n}, By_{n}-By^*\rangle\\
&\quad+2\gamma_n^2\left(\|(I-P_Q)y_{n}\|+\|B^*(Ax_{n}-By_{n})\|^2\right).
\end{aligned}
\end{equation}
On the other hand, we get
\[
2\langle Ax_{n}-By_{n},By_{n}-By^*\rangle-2\langle Ax_{n}-By_{n}, Ax_{n}-Ax^*\rangle=-2\|Ax_{n}-By_{n}\|^2.
\]
From the definition of $\gamma_n$, we have $\gamma_n(1-\gamma_n)>0$ and
\[
\gamma_n\|Ax_n-By_n\|^2-\gamma_n^2(\|A^*(Ax_{n}-By_{n})\|^2+\|B^*(Ax_{n}-By_{n})\|^2)\geq0.
\]
Combine formulas \eqref{zn}, \eqref{sn} and condition {{(R1)}} to get
\begin{equation}\label{wxy}
\begin{aligned}
&\quad\|\widehat{x}_{n}-x^*\|^2+\|\widehat{y}_{n}-y^*\|^2\\
&\leq \|x_{n}-x^*\|^2+\|y_{n}-y^*\|^2-2\gamma_n(1-\gamma_n)\left(\|(I-P_C)x_n\|^2+\|(I-P_Q)y_n\|^2\right)\\
&\quad -2\gamma_n\left(\|Ax_n-By_n\|^2-\gamma_n\|A^*(Ax_{n}-By_{n})\|^2-\gamma_n\|B^*(Ax_{n}-By_{n})\|^2\right)\\
&\leq \|x_{n}-x^*\|^2+\|y_{n}-y^*\|^2.
\end{aligned}
\end{equation}
For the convenience of writing, set
\[
\begin{aligned}
\Phi_n&=2\gamma_n(1-\gamma_n)\left(\|(I-P_C)x_n\|^2+\|(I-P_Q)y_n\|^2\right)+2\gamma_n\|Ax_n-By_n\|^2\\
&\quad -2\gamma_n^2\left(\|A^*(Ax_{n}-By_{n})\|^2+\|B^*(Ax_{n}-By_{n})\|^2\right).
\end{aligned}
\]
Obviously, $\Phi_n\geq0$. In addition, using the convexity of the squared norm and \eqref{wxy}, we have
\[
\|x_{n+1}-x^*\|^2+\|y_{n+1}-y^*\|^2\leq\delta_n\left(\|u_{n}-x^*\|^2+\|v_{n}-y^*\|^2\right)+(1-\delta_n)\left(\|x_{n}-x^*\|^2+\|y_{n}-y^*\|^2\right).
\]
Since $\{u_n\}$ and $\{v_n\}$ are convergence sequences, there exists a non-negative constant $G$ such that $\sup_{n\geq 0}\{\|u_{n}-x^*\|^2, \|v_{n}-y^*\|^2\}\leq {G}/{2}$. Let $\theta_{n}=\|x_{n}-x^*\|^2+\|y_{n}-y^*\|^2$. The above formula can be converted to
\begin{equation}\label{xy3}
\begin{aligned}
\theta_{n+1}\leq \delta_n\left(\|u_{n}-x^*\|^2+\|v_{n}-y^*\|^2\right)+(1-\delta_{n})\theta_{n}\leq \max\{G, \theta_n\}\leq \cdots \leq \max\{G, \theta_0\}.
\end{aligned}
\end{equation}
This implies that $\{\theta_{n}\}$ is bounded, that is, the sequences $\{x_n\}$ and $\{y_n\}$ are bounded. From property {(P5)}, we have
\[
\begin{aligned}
\|x_{n+1}-x^*\|^2&=\|\delta_n u_n+(1-\delta_n)\widehat{x}_n-x^*\|^2\\
&\leq \|(1-\delta_n)\left(\widehat{x}_n-x^*\right)\|^2+2\delta_n\langle u_n-x^*,x_{n+1}-x^*\rangle\\
&\leq (1-\delta_n)\|\widehat{x}_n-x^*\|^2+2\delta_n\langle u_n-x^*,x_{n+1}-x^*\rangle,
\end{aligned}
\]
\[
\|y_{n+1}-y^*\|^2\leq (1-\delta_n)\|\widehat{y}_n-y^*\|^2+2\delta_n\langle v_n-y^*,y_{n+1}-y^*\rangle.
\]
Combining the above two inequalities and \eqref{wxy}, we get
\begin{equation}\label{xyxy}
\begin{aligned}
\|x_{n+1}-x^*\|^2+\|y_{n+1}-y^*\|^2&\leq(1-\delta_n)\left(\|x_n-x^*\|^2+\|y_n-y^*\|^2\right)-(1-\delta_n)\Phi_n\\
&\quad+2\delta_n\left(\langle u_n-x^*,x_{n+1}-x^*\rangle+\langle v_n-y^*,y_{n+1}-y^*\rangle\right).
\end{aligned}
\end{equation}
For each $n\geq 0$, set
\begin{align*}
&\theta_n=\|x_n-x^*\|^2+\|y_n-y^*\|^2,\ \eta_n=(1-\delta_n)\Phi_n;\\
&\tau_n=2\left(\langle u_n-x^*,x_{n+1}-x^*\rangle+\langle v_n-y^*,y_{n+1}-y^*\rangle\right);\\
&\zeta_n=2\delta_n\left(\langle u_n-x^*,x_{n+1}-x^*\rangle+\langle v_n-y^*,y_{n+1}-y^*\rangle\right).
\end{align*}
Then, the above formula \eqref{xyxy} is reduced to the following inequalities:
\[
\theta_{n+1}\leq (1-\delta_{n})\theta_{n}+\delta_n\tau_n,\ n\geq 0,
\text{ and } \quad \theta_{n+1}\leq \theta_{n}-\eta_n+\zeta_n,\ n\geq 0.
\]
By the boundedness of $\{u_n\}$, $\{v_n\}$, $\{x_n\}$ and $\{y_n\}$, and condition {{(R2)}}, we see that $\underset{n\rightarrow \infty}{\lim}\zeta_n=0$ and $\sum_{n=0}^\infty\delta_n=\infty$. By virtue of Lemma \ref{lem2.5}, this proof remains to show that $\underset{k\rightarrow \infty}{\lim}\eta_{n_k}=0$ implies $\underset{k\rightarrow \infty}{\limsup}\ \tau_{n_k}\leq0$ for any subsequence of real numbers $\{n_k\}$ of $\{n\}$. Let $\{\eta_{n_k}\}$ be a any subsequence of $\{\eta_n\}$ such that $\underset{k\rightarrow \infty}{\lim}\eta_{n_k}=0$. If $Ax_{n}\neq By_{n}$, it follows from condition (R1) and $\Phi_n$ that
\begin{equation}\label{limxy1}
\underset{k\rightarrow \infty}{\lim}\|(I-P_C)x_{n_k}\|=\underset{k\rightarrow \infty}{\lim}\|(I-P_Q)y_{n_k}\|=\underset{k\rightarrow \infty}{\lim}\|Ax_{n_k}-By_{n_k}\|=0.
\end{equation}
From the boundedness of $\{x_n\}$ and $\{y_n\}$, there exists two subsequences $\{x_{n_{k_j}}\}$ of $\{x_{n_k}\}$ and $\{y_{n_{k_j}}\}$ of $\{y_{n_k}\}$ such that $x_{n_{k_j}}\rightharpoonup \bar{x}$, $y_{n_{k_j}}\rightharpoonup \bar{y}$ and
\[
\underset{k\rightarrow \infty}{\limsup}\ \langle u_{n_k}-x^*,x_{n_k}-x^*\rangle=\underset{j\rightarrow \infty}{\lim}\langle u_{n_{k_j}}-x^*, x_{n_{k_j}}-x^*\rangle,
\]
\[
\underset{k\rightarrow \infty}{\limsup}\ \langle v_{n_k}-y^*,y_{n_k}-y^*\rangle=\underset{j\rightarrow \infty}{\lim}\langle v_{n_{k_j}}-y^*, y_{n_{k_j}}-y^*\rangle.
\]
Since $I-P_C$ and $I-P_Q$ are demiclosed at $0$, from \eqref{limxy1}, we have $\bar{x}\in C$ and $\bar{y}\in Q$. In addition, it follows from the bounded linearity of $A$ and $B$ that $Ax_{n_{k_j}}-By_{n_{k_j}}\rightharpoonup A\bar{x}-B\bar{y}$. Using the weak lower semicontinuity of the squared norm, we have $
\|A\bar{x}-B\bar{y}\|^2\leq \underset{j\rightarrow \infty}{\liminf}\ \|Ax_{n_{k_j}}-By_{n_{k_j}}\|^2=0$,
which implies that $(\bar{x},\bar{y})\in \Omega$. On the other hand, if $Ax_{n}=By_{n}$, it is clear that we can also get the same result as above. In addition, from the property of projection and the strong convergence property of $\{u_n\}$ and $\{v_n\}$, it follows that
\begin{equation}\label{xnk1}
\underset{k\rightarrow \infty}{\limsup}\ \langle u_{n_k}-x^*,x_{n_k}-x^*\rangle=\underset{j\rightarrow \infty}{\lim}\langle u_{n_{k_j}}-x^*, x_{n_{k_j}}-x^*\rangle=\langle u-x^*,\bar{x}-x^*\rangle\leq 0,
\end{equation}
\begin{equation}\label{xnk2}
\underset{k\rightarrow \infty}{\limsup}\ \langle v_{n_k}-y^*,y_{n_k}-y^*\rangle=\underset{j\rightarrow \infty}{\lim}\langle v_{n_{k_j}}-y^*, y_{n_{k_j}}-y^*\rangle=\langle v-y^*,\bar{y}-y^*\rangle\leq 0.
\end{equation}
According to condition {(R1)} and \eqref{limxy1}, we have
\[
\begin{aligned}
\|\widehat{x}_{n_k}-x_{n_k}\|\leq \gamma_n(\|(I-P_C)x_{{n_k}}\|+\|A\|\|Ax_{{n_k}}-By_{{n_k}}\|)\rightarrow 0,
\end{aligned}
\]
\[
\|\widehat{y}_{n_k}-y_{n_k}\|\leq \gamma_n(\|(I-P_Q)y_{{n_k}}\|+\|B\|\|Ax_{{n_k}}-By_{{n_k}}\|)\rightarrow 0.
\]
Further, we have
\begin{equation}\label{xx1}
\|x_{{n_k}+1}-x_{n_k}\|\leq\delta_{n_k}\|u_{n_k}-x_{n_k}\|+(1-\delta_{n_k})\|\widehat{x}_{n_k}-x_{n_k}\|\rightarrow 0,
\end{equation}
\begin{equation}\label{xx2}
\|y_{{n_k}+1}-y_{n_k}\|\leq\delta_{n_k}\|v_{n_k}-y_{n_k}\|+(1-\delta_{n_k})\|\widehat{y}_{n_k}-y_{n_k}\|\rightarrow 0.
\end{equation}
From the formula \eqref{xnk1}, \eqref{xnk2}, \eqref{xx1} and \eqref{xx2}, we have $\underset{k\rightarrow \infty}{\limsup}\ \langle u_{n_k}-x^*,x_{n_k+1}-x^*\rangle\leq 0$ and  $\underset{k\rightarrow \infty}{\limsup}\ \langle v_{n_k}-y^*,y_{n_k+1}-y^*\rangle\leq 0.$
This implies that $\underset{k\rightarrow \infty}{\limsup}\ \tau_{n_k}\leq 0$. By virtue of {Lemma \ref{lem2.5}}, we obtain $\underset{n\rightarrow \infty}{\lim}\theta_n=0$, which implies that $(x_n,y_n)\rightarrow (x^*,y^*)$. \qed
\end{proof}

\begin{remark}\label{mark1}
	\begin{itemize}
	\item The sequences $\{u_n\}$ and $\{v_n\}$ in {{Theorem \ref{sep1}}} are easily chosen, for example, (1) the monotonically decreasing sequence $u_n=\frac{n^2}{(n+1)^2}u$; (2) the monotonically increasing sequence $u_n=\frac{(n+1)^2}{n^2}u$; (3) the non-monotonically convergent sequence $u_n=\frac{2n+(-1)^n}{2n}u$.
	\item In particular, when the sequences $\{u_n\}$ and $\{v_n\}$ are constant sequences, that is, $u_n\equiv u$ and $v_n\equiv v$, the Halpern-type algorithm is equal to the classical Halpern algorithm {\cite{halpern1967fixed}}.
	\end{itemize}
\end{remark}

Therefore, from Remark \ref{mark1}, we have the following corollary and the corresponding iterative algorithm. Assumptions (A1)-(A3) are  established, algorithm \eqref{CQ} in Theorem \ref{sep1} can be degraded to the following Halpern algorithm
\begin{equation}\label{HCQ}
		\left\{\begin{aligned}
&{\widehat{x}_{n}=x_{n}-\gamma_n\left[(I-P_C)x_{n}+A^*(Ax_{n}-By_{n})\right],}\\
&{x_{n+1}=\delta_n u+(1-\delta_n)\widehat{x}_{n},}\\
&{\widehat{y}_{n}=y_{n}-\gamma_n\left[(I-P_Q)y_{n}-B^*(Ax_{n}-By_{n})\right],}\\
&{y_{n+1}=\delta_n v+(1-\delta_n)\widehat{y}_{n}, n\geq 0.}
\end{aligned}\right.
\end{equation}

\begin{corollary}\label{cor1}
If conditions {(R1)-(R2)} are also satisfied, the iterative sequence $\{(x_n, y_n)\}$ generated by algorithm \eqref{HCQ}  strongly converges to $P_\Omega(u,v)\in \Omega$.
\end{corollary}

\subsection{Self-adaptive relaxed Halpern-type CQ algorithm (SRHCQA)}
Here, we consider using the level sets of two convex functions $\mathfrak{f}:H_1\rightarrow R$ and $\mathfrak{g}:H_2\rightarrow R$ instead of closed convex sets $C$ and $Q$ in Theorem \ref{sep1}, i.e., $C=\{x\in H_1: \mathfrak{f}(x)\leq 0\},\ Q=\{y\in H_2: \mathfrak{g}(y)\leq0\}$. For solving SEP \eqref{sep}, we construct the corresponding closed convex sets as follows.
\[
C_n=\{x\in H_1|\mathfrak{f}(x_n)+\langle\xi_n,x-x_n\rangle\leq0\},\ \xi_n\in\partial \mathfrak{f}(x_n),
\]
\[
Q_n=\{y\in H_2|\mathfrak{g}(y_n)+\langle\eta_n,y-y_n\rangle\leq0\},\ \eta_n\in\partial \mathfrak{g}(y_n).
\]
Note here that $\partial \mathfrak{f}$ and $\partial \mathfrak{g}$ are bounded sets. It is obvious that $C\subset C_n$ and $Q\subset Q_n$. Under the above conditions, using the method of Halpern-type algorithm to promote the relaxed alternating CQ algorithm, the iterative sequence $\{(x_n, y_n)\}$ of the split equality problem \eqref{sep} is generated by the following recursive procedure
\begin{equation}\label{RCQ}
		\left\{\begin{aligned}
&{\widehat{x}_{n}=x_{n}-\gamma_n\left[(I-P_{C_n})x_{n}+A^*(Ax_{n}-By_{n})\right],}\\
&{\widehat{y}_{n}=y_{n}-\gamma_n\left[(I-P_{Q_n})y_{n}-B^*(Ax_{n}-By_{n})\right],}\\
&{x_{n+1}=\delta_n u_n+(1-\delta_n)\widehat{x}_{n},}\\
&{y_{n+1}=\delta_n v_n+(1-\delta_n)\widehat{y}_{n},  n\geq 0.}
\end{aligned}\right.
\end{equation}
The corresponding parameters satisfy the following restrictions
\begin{description}
    \item[\textbf{(R3)}] If $Ax_{n}\neq By_{n}$, the self-adaptive stepsize
    \[
        \gamma_n=\alpha_n\min\left\{1,\frac{\|Ax_n-By_n\|^2}{\|A^*(Ax_{n}-By_{n})\|^2+\|B^*(Ax_{n}-By_{n})\|^2}\right\}\ \text{with}\ \alpha_n\in(0,1).
        \]
Otherwise, the self-adaptive stepsize $\gamma_n=\alpha_n$;
	\item[\textbf{(R4)}] $\{\delta_n\}$ is real number sequence in $(0,1)$ such that $\lim_{n\rightarrow \infty}\delta_n=0$ and $\sum_{n=0}^{\infty}\delta_n=\infty$.
\end{description}

\begin{theorem}\label{sep2}
Given the assumptions {{(A1)-(A3)}} and conditions {{(R3)-(R4)}}, the iterative sequence $\{(x_n, y_n)\}$ generated by algorithm \eqref{RCQ} converges strongly to $P_\Omega(u,v)\in \Omega$.
\end{theorem}

\begin{proof}
Take $(x^*, y^*)=P_\Omega(u,v)\in \Omega$, that is, $x^*\in C$, $y^*\in Q$ and $Ax^*=By^*$. As similar proof in {Theorem \ref{sep1}}, set $\theta_{n}=\|x_{n}-x^*\|^2+\|y_{n}-y^*\|^2$. Similarly we obtain that the sequences $\{x_n\}$ and $\{y_n\}$ are bounded. On the other hand, set
\[
\begin{aligned}
\Psi_n&=2\gamma_n(1-\gamma_n)(\|(I-P_{C_n})x_n\|^2+\|(I-P_{Q_n})y_n\|^2)+2\gamma_n\|Ax_n-By_n\|^2\\
&\quad -2\gamma_n^2\left(\|A^*(Ax_{n}-By_{n})\|^2+\|B^*(Ax_{n}-By_{n})\|^2\right).
\end{aligned}
\]
From the definition of the self-adaptive stepsize $\gamma_n$, we have $\gamma_n(1-\gamma_n)>0$ and
\[
\gamma_n\|Ax_n-By_n\|^2-\gamma_n^2(\|A^*(Ax_{n}-By_{n})\|^2+\|B^*(Ax_{n}-By_{n})\|^2)\geq0.
\]
Obviously, $\Psi_n\geq0$. Similarly, we have
\[
\begin{aligned}
\|x_{n+1}-x^*\|^2+\|y_{n+1}-y^*\|^2&\leq(1-\delta_n)\left(\|x_n-x^*\|^2+\|y_n-y^*\|^2\right)-(1-\delta_n)\Psi_n\\
&+2\delta_n\left(\langle u_n-x^*,x_{n+1}-x^*\rangle+\langle v_n-y^*,y_{n+1}-y^*\rangle\right).
\end{aligned}
\]
For each $n\geq 0$, we get
\begin{align*}
&\theta_n=\|x_n-x^*\|^2+\|y_n-y^*\|,\ \eta_n=(1-\delta_n)\Psi_n;\\
&\tau_n=2\left(\langle u_n-x^*,x_{n+1}-x^*\rangle+\langle v_n-y^*,y_{n+1}-y^*\rangle\right);\\
&\zeta_n=2\delta_n\left(\langle u_n-x^*,x_{n+1}-x^*\rangle+\langle v_n-y^*,y_{n+1}-y^*\rangle\right).
\end{align*}
Naturally, we have the following inequalities
\[
\theta_{n+1}\leq (1-\delta_{n})\theta_{n}+\delta_n\tau_n, n\geq 0,
\text{ and }\quad
\theta_{n+1}\leq \theta_{n}-\eta_n+\zeta_n, n\geq 0.
\]
By the boundedness of $\{u_n\}$, $\{v_n\}$, $\{x_n\}$ and $\{y_n\}$, and condition {(R4)}, we see that $\underset{n\rightarrow \infty}{\lim}\zeta_n=0$ and $\sum_{n=0}^\infty\delta_n=\infty$. By virtue of {Lemma \ref{lem2.5}}, this proof remains to show that $\underset{k\rightarrow \infty}{\lim}\eta_{n_k}=0$ implies $\underset{k\rightarrow \infty}{\limsup}\ \tau_{n_k}\leq0$ for any subsequence of real numbers $\{n_k\}$ of $\{n\}$. Let $\{\eta_{n_k}\}$ be a any subsequence of $\{\eta_n\}$ such that $\underset{k\rightarrow \infty}{\lim}\eta_{n_k}=0$. If $Ax_{n}\neq By_{n}$, it follows from condition {(R3)} and $\Psi_n$ that
\begin{equation}\label{limxy11}
\underset{k\rightarrow \infty}{\lim}\|(I-P_{C_n})x_{n_k}\|=\underset{k\rightarrow \infty}{\lim}\|(I-P_{Q_n})y_{n_k}\|=\underset{k\rightarrow \infty}{\lim}\|Ax_{n_k}-By_{n_k}\|=0.
\end{equation}
Using the boundedness of $\{x_n\}$ and $\{y_n\}$, there exists two sequences $\{x_{n_{k_j}}\}$ of $\{x_{n_k}\}$ and $\{y_{n_{k_j}}\}$ of $\{y_{n_k}\}$ such that $x_{n_{k_j}}\rightharpoonup \bar{x}$, $y_{n_{k_j}}\rightharpoonup \bar{y}$ and
\[
\underset{k\rightarrow \infty}{\limsup}\ \langle u_{n_k}-x^*,x_{n_k}-x^*\rangle=\underset{j\rightarrow \infty}{\lim}\langle u_{n_{k_j}}-x^*, x_{n_{k_j}}-x^*\rangle,
\]
\[
\underset{k\rightarrow \infty}{\limsup}\ \langle v_{n_k}-y^*,y_{n_k}-y^*\rangle=\underset{j\rightarrow \infty}{\lim}\langle v_{n_{k_j}}-y^*, y_{n_{k_j}}-y^*\rangle.
\]
Since $\partial \mathfrak{f}$ is bounded on bounded sets, there exists a constant $\vartheta> 0$ such that $\|\xi_{n_k}\|\leq \vartheta$, $\forall k\geq 0$. Using the definition of $C_{n}$ and \eqref{limxy11} to get
$\mathfrak{f}(x_{n_k})\leq \langle \xi_{n_k},x_{n_k}-P_{C_{n_k}}x_{n_k}\rangle \leq \vartheta\|x_{n_k}-P_{C_{n_k}}x_{n_k}\|\rightarrow 0$ as $k\rightarrow \infty$. By the function $\mathfrak{f}$ has weak lower semi-continuity, $\mathfrak{f}(\bar{x})\leq \underset{j\rightarrow \infty}{\liminf}\ \mathfrak{f}(x_{n_{k_j}})\leq 0$ can be obtained. So we have $\bar{x}\in C$. Similarly, we obtain $\bar{y}\in Q$. In addition, it follows from the bounded linearity of $A$ and $B$ that $Ax_{n_{k_j}}-By_{n_{k_j}}\rightharpoonup A\bar{x}-B\bar{y}$. By virtue of the weak lower semicontinuity of the squared norm, we have $\|A\bar{x}-B\bar{y}\|^2\leq \underset{j\rightarrow \infty}{\liminf}\ \|Ax_{n_{k_j}}-By_{n_{k_j}}\|^2=0$, which implies that $(\bar{x},\bar{y})\in \Omega$. On the other hand, if $Ax_{n}=By_{n}$, it is clearly that we can also get the same result as above. Last, using the proof process in {Theorem \ref{sep1}} and {Lemma \ref{lem2.5}}, we obtain $\underset{n\rightarrow \infty}{\lim}\theta_n=0$, this implies that $(x_n,y_n)\rightarrow (x^*,y^*)$. \qed
\end{proof}

According to Remark \ref{mark1} (ii), in the case where assumptions (A1)-(A3) are established, algorithm \eqref{RCQ} in Theorem \ref{sep2} can be degraded to the following Halpern algorithm
\begin{equation}\label{RCQ1}
		\left\{\begin{aligned}
&{\widehat{x}_{n}=x_{n}-\gamma_n\left[(I-P_{C_n})x_{n}+A^*(Ax_{n}-By_{n})\right],}\\
&{\widehat{y}_{n}=y_{n}-\gamma_n\left[(I-P_{Q_n})y_{n}-B^*(Ax_{n}-By_{n})\right],}\\
&{x_{n+1}=\delta_n u+(1-\delta_n)\widehat{x}_{n},}\\
&{y_{n+1}=\delta_n v+(1-\delta_n)\widehat{y}_{n},  n\geq 0.}
\end{aligned}\right.
\end{equation}

\begin{corollary}\label{cor2}
If conditions {{(R3)-(R4)}} are established, the iterative sequence $\{(x_n, y_n)\}$ generated by algorithm \eqref{RCQ1} converges strongly to $P_\Omega(u,v)\in \Omega$.
\end{corollary}

\section{Viscosity-type CQ algorithms}\label{se4}
In this section, we propose two viscosity-type algorithms to approximate a solution of the split equality problems \eqref{sep}, and assume that the solution set of SEP \eqref{sep} is nonempty, i.e., $\Omega=\{(x^*,y^*)\in C\times Q| Ax^*=By^*\}\neq \emptyset$. In addition, the following assumptions are presupposed.
\begin{description}
	\item[\textbf{(V1)}] $H_1$, $H_2$ and $H_3$ are Hilbert spaces, and $C\subset H_1$, $Q\subset H_2$ are nonempty closed convex subsets;
	\item[\textbf{(V2)}] $A:H_1\rightarrow H_3$, $B:H_2\rightarrow H_3$ are bounded linear operators, and $A^*$, $B^*$ are the adjoint operators of $A$, $B$, respectively;
    \item[\textbf{(V3)}] $f:H_1\rightarrow H_1$ and $g:H_2\rightarrow H_2$ are contraction mappings with coefficient $\lambda_1\in [0,{1}/{\sqrt{2}})$, $\lambda_2\in [0,{1}/{\sqrt{2}})$, respectively.
\end{description}

\subsection{Self-adaptive viscosity-type CQ algorithm (SVCQA)}
According to the mentioned viscosity-type algorithm in Moudafi \cite{moudafi2000viscosity}, the iterative sequence $\{(x_n, y_n)\}$ of the split equality problem \eqref{sep} is generated by the following recursive procedure
\begin{equation}\label{VCQ}
		\left\{\begin{aligned}
&{\widehat{x}_{n}=x_{n}-\gamma_n\left[(I-P_C)x_{n}+A^*(Ax_{n}-By_{n})\right],}\\
&{x_{n+1}=\delta_n f(\widehat{x}_n)+(1-\delta_n)\widehat{x}_{n},}\\
&{\widehat{y}_{n}=y_{n}-\gamma_n\left[(I-P_Q)y_{n}-B^*(Ax_{n}-By_{n})\right],}\\
&{y_{n+1}=\delta_n g(\widehat{y}_n)+(1-\delta_n)\widehat{y}_{n}, n\geq 0.}
\end{aligned}\right.
\end{equation}
\noindent The corresponding parameters satisfy the following restrictions
\begin{description}
    \item[\textbf{(E1)}] If $Ax_{n}\neq By_{n}$, the self-adaptive stepsize
    \[\gamma_n=\alpha_n\min\left\{1,\frac{\|Ax_n-By_n\|^2}{\|A^*(Ax_{n}-By_{n})\|^2+\|B^*(Ax_{n}-By_{n})\|^2}\right\}\ \text{with}\ \alpha_n\in(0,1).
    \]
    Otherwise, the self-adaptive stepsize $\gamma_n=\alpha_n$;
    \item[\textbf{(E2)}] $\{\delta_n\}$ is real number sequence in $(0,1)$ such that $\lim_{n\rightarrow \infty}\delta_n=0$ and  $\sum_{n=0}^{\infty}\delta_n=\infty$.
\end{description}

\begin{theorem}\label{sep3}
Given the assumptions {{(V1)-(V3)}} and conditions {{(E1)-(E2)}}. The iterative sequence $\{(x_n, y_n)\}$ generated by algorithm \eqref{VCQ} converges strongly to $(x^*,y^*)=P_\Omega(f(x^*),g(y^*))$.
\end{theorem}

\begin{proof}
Take $(x^*, y^*)=P_\Omega(f(x^*),g(y^*))\in \Omega$, that is, $x^*\in C$, $y^*\in Q$ and $Ax^*=By^*$. From the proof of {Theorem \ref{sep1}}, we can get
\[
\begin{aligned}
\Phi_n&=2\gamma_n(1-\gamma_n)(\|(I-P_{C})x_n\|^2+\|(I-P_{Q})y_n\|^2)+2\gamma_n\|Ax_n-By_n\|^2\\
&\quad -2\gamma_n^2\left(\|A^*(Ax_{n}-By_{n})\|^2+\|B^*(Ax_{n}-By_{n})\|^2\right).
\end{aligned}
\]
From the definition of the self-adaptive stepsize $\gamma_n$, we know $\Phi_n\geq0$. We can also get
\begin{equation}\label{fxy}
\begin{aligned}
\|\widehat{x}_{n}-x^*\|^2+\|\widehat{y}_{n}-y^*\|^2=\|x_{n}-x^*\|^2+\|y_{n}-y^*\|^2-\Phi_n\leq \|x_{n}-x^*\|^2+\|y_{n}-y^*\|^2.
\end{aligned}
\end{equation}
Set $\lambda=\max\{\lambda_1, \lambda_2\}$. Using the convexity of the squared norm and \eqref{fxy}, we have
\begin{equation}\label{xyzsv}
\begin{aligned}
&\quad \|x_{n+1}-x^*\|^2+\|y_{n+1}-y^*\|^2\\
&\leq\delta_n\left(\|f(\widehat{x}_{n})-x^*\|^2+\|g(\widehat{y}_{n})-y^*\|^2\right)+(1-\delta_n)\left(\|\widehat{x}_{n}-x^*\|^2+\|\widehat{y}_{n}-y^*\|^2\right)\\
&\leq 2\delta_n\left(\|f(\widehat{x}_{n})-f(x^*)\|^2+\|f(x^*)-x^*\|^2+\|g(\widehat{y}_{n})-g(y^*)\|^2+\|g(y^*)-y^*\|^2\right)\\
&\quad +(1-\delta_n)\left(\|\widehat{x}_{n}-x^*\|^2+\|\widehat{y}_{n}-y^*\|^2\right)\\
&\leq 2\delta_n(\lambda_1^2\|\widehat{x}_n-x^*\|^2+\lambda_2^2\|\widehat{y}_n-y^*\|^2)+2\delta_n(\|f(x^*)-x^*\|^2+\|g(y^*)-y^*\|^2)\\
&\quad +(1-\delta_n)\left(\|\widehat{x}_{n}-x^*\|^2+\|\widehat{y}_{n}-y^*\|^2\right)\\
&=(1-\delta_n(1-2\lambda^2))\left(\|\widehat{x}_{n}-x^*\|^2+\|\widehat{y}_{n}-y^*\|^2\right)+2\delta_n(\|f(x^*)-x^*\|^2+\|g(y^*)-y^*\|^2)\\
&\leq (1-\delta_n(1-2\lambda^2))\left(\|x_{n}-x^*\|^2+\|y_{n}-y^*\|^2\right)+2\delta_n(\|f(x^*)-x^*\|^2+\|g(y^*)-y^*\|^2)\\
&\quad -(1-\delta_n(1-2\lambda^2))\Phi_n.
\end{aligned}
\end{equation}
Let $\theta_{n}=\|x_{n}-x^*\|^2+\|y_{n}-y^*\|^2$. Since $\lambda_1, \lambda_2\in [0,{1}/{\sqrt{2}})$ and by condition {(E1)}, the formula \eqref{xyzsv} can be converted to
\begin{equation}\label{xyzsv1}
\begin{aligned}
\theta_{n+1}&\leq (1-\delta_n(1-2\lambda^2))\theta_{n}+2\delta_n\left(\|f(x^*)-x^*\|^2+\|g(y^*)-y^*\|^2\right) -(1-\delta_n(1-2\lambda^2))\Phi_n\\
&\leq (1-\delta_n(1-2\lambda^2))\theta_{n}+2\delta_n(1-2\lambda^2)\frac{\|f(x^*)-x^*\|^2+\|g(y^*)-y^*\|^2}{1-2\lambda^2}\\
&\leq \max\{\theta_n, \frac{2(\|f(x^*)-x^*\|^2+\|g(y^*)-y^*\|^2)}{1-2\lambda^2}\}\\
&\leq \cdots \leq \max\{\theta_0, \frac{2(\|f(x^*)-x^*\|^2+\|g(y^*)-y^*\|^2)}{1-2\lambda^2}\}.
\end{aligned}
\end{equation}
This implies that $\{\theta_{n}\}$ is bounded, that is, the sequences $\{x_n\}$ and $\{y_n\}$ are bounded. On the other hand, using the property {(P5)}, we can obtain
\[
\begin{aligned}
\|x_{n+1}-x^*\|^2&=\|\delta_n f(\widehat{x}_n)+(1-\delta_n)\widehat{x}_n-x^*\|^2\\
&=\|\delta_n (f(\widehat{x}_n)-f(x^*))+\delta_n (f(x^*)-x^*)+(1-\delta_n)(\widehat{x}_n-x^*)\|^2\\
&\leq \|\delta_n (f(\widehat{x}_n)-f(x^*))+(1-\delta_n)(\widehat{x}_n-x^*)\|^2+2\delta_n \langle f(x^*)-x^*, x_{n+1}-x^*\rangle\\
&\leq \delta_n\|f(\widehat{x}_n)-f(x^*)\|^2+(1-\delta_n)\|\widehat{x}_n-x^*\|^2+2\delta_n\langle f(x^*)-x^*,x_{n+1}-x^*\rangle\\
&\leq (1-\delta_n(1-\lambda_1^2))\|\widehat{x}_n-x^*\|^2+2\delta_n\langle f(x^*)-x^*,x_{n+1}-x^*\rangle.
\end{aligned}
\]
Similarly,
\[
\|y_{n+1}-y^*\|^2\leq (1-\delta_n(1-\lambda_2^2))\|\widehat{y}_n-y^*\|^2+2\delta_n\langle g(y^*)-y^*,y_{n+1}-y^*\rangle.
\]
According to the above formulas, we get
\[
\begin{aligned}
\|x_{n+1}-x^*\|^2+\|y_{n+1}-y^*\|^2&\leq(1-\delta_n(1-\lambda^2))\left(\|\widehat{x}_n-x^*\|^2+\|\widehat{y}_n-y^*\|^2\right)\\
&\quad+2\delta_n\left(\langle f(x^*)-x^*,x_{n+1}-x^*\rangle+\langle g(y^*)-y^*,y_{n+1}-y^*\rangle\right)\\
&\leq(1-\delta_n(1-\lambda^2))\left(\|x_n-x^*\|^2+\|y_n-y^*\|^2\right)-(1-\delta_n(1-\lambda^2))\Phi_n\\
&\quad+2\delta_n\left(\langle f(x^*)-x^*,x_{n+1}-x^*\rangle+\langle g(y^*)-y^*,y_{n+1}-y^*\rangle\right).
\end{aligned}
\]
For each $n\geq 0$, we also set
\begin{align*}
\theta_n&=\|x_n-x^*\|^2+\|y_n-y^*\|^2,\ \eta_n=(1-\delta_n(1-\lambda^2))\Phi_n;\\
\tau_n&=2\left(\langle f(x^*)-x^*,x_{n+1}-x^*\rangle+\langle g(y^*)-y^*,y_{n+1}-y^*\rangle\right)/(1-\lambda^2);\\
\zeta_n&=2\delta_n\left(\langle f(x^*)-x^*,x_{n+1}-x^*\rangle+\langle g(y^*)-y^*,y_{n+1}-y^*\rangle\right).
\end{align*}
Then, the above formula is reduced to the following inequalities
\[
\theta_{n+1}\leq (1-\delta_{n})\theta_{n}+\delta_n\tau_n, n\geq 0
\text{ and }\
\theta_{n+1}\leq \theta_{n}-\eta_n+\zeta_n, n\geq 0.
\]
By the boundedness of $\{x_n\}$ and $\{y_n\}$, and condition {(E2)}, we see that $\underset{n\rightarrow \infty}{\lim}\zeta_n=0$ and $\sum_{n=0}^\infty\delta_n=\infty$. By virtue of{ Lemma \ref{lem2.5}}, this proof remains to show that $\underset{k\rightarrow \infty}{\lim}\eta_{n_k}=0$ implies $\underset{k\rightarrow \infty}{\limsup}\ \tau_{n_k}\leq0$ for any subsequence of real numbers $\{n_k\}$ of $\{n\}$. Let $\{\eta_{n_k}\}$ be a any subsequence of $\{\eta_n\}$ such that $\underset{k\rightarrow \infty}{\lim}\eta_{n_k}=0$. If $Ax_{n}\neq By_{n}$, it follows from condition (E1) and $\Phi_n$ that
\begin{equation}\label{limxy1v}
\underset{k\rightarrow \infty}{\lim}\|(I-P_C)x_{n_k}\|=\underset{k\rightarrow \infty}{\lim}\|(I-P_Q)y_{n_k}\|=\underset{k\rightarrow \infty}{\lim}\|Ax_{n_k}-By_{n_k}\|=0.
\end{equation}
By the boundedness of $\{x_n\}$ and $\{y_n\}$, there exists two sequences $\{x_{n_{k_j}}\}$ of $\{x_{n_k}\}$ and $\{y_{n_{k_j}}\}$ of $\{y_{n_k}\}$ such that $x_{n_{k_j}}\rightharpoonup \bar{x}$, $y_{n_{k_j}}\rightharpoonup \bar{y}$ and
\[
\underset{k\rightarrow \infty}{\limsup}\ \langle f(x^*)-x^*,x_{n_k}-x^*\rangle=\underset{j\rightarrow \infty}{\lim}\langle f(x^*)-x^*, x_{n_{k_j}}-x^*\rangle,
\]
\[
\underset{k\rightarrow \infty}{\limsup}\ \langle g(y^*)-y^*,y_{n_k}-y^*\rangle=\underset{j\rightarrow \infty}{\lim}\langle g(y^*)-y^*, y_{n_{k_j}}-y^*\rangle.
\]
Since $I-P_C$ and $I-P_Q$ are demiclosed at $0$, from \eqref{limxy1v}, we have $\bar{x}\in C$ and $\bar{y}\in Q$. In addition, it follows from bounded linearity of $A$ and $B$ that $Ax_{n_{k_j}}-By_{n_{k_j}}\rightharpoonup A\bar{x}-B\bar{y}$. Using the weak lower semicontinuity of the squared norm implies $\|A\bar{x}-B\bar{y}\|^2\leq \underset{j\rightarrow \infty}{\liminf}\ \|Ax_{n_{k_j}}-By_{n_{k_j}}\|^2=0$,
which implies that $(\bar{x},\bar{y})\in \Omega$. On the other hand, if $Ax_{n}=By_{n}$, it is clearly that we can also get the same result as above. In addition, from the property of projection, it follows that
\[
\underset{k\rightarrow \infty}{\limsup}\ \langle f(x^*)-x^*,x_{n_k}-x^*\rangle=\underset{j\rightarrow \infty}{\lim}\langle f(x^*)-x^*, x_{n_{k_j}}-x^*\rangle=\langle f(x^*)-x^*,\bar{x}-x^*\rangle\leq 0,
\]
\[
\underset{k\rightarrow \infty}{\limsup}\ \langle g(y^*)-y^*,y_{n_k}-y^*\rangle=\underset{j\rightarrow \infty}{\lim}\langle g(y^*)-y^*, y_{n_{k_j}}-y^*\rangle=\langle g(y^*)-y^*,\bar{y}-y^*\rangle\leq 0.
\]
According to condition {(E1)} and \eqref{limxy1v}, we have
\[
\begin{aligned}
\|\widehat{x}_{n_k}-x_{n_k}\|\leq \gamma_n(\|(I-P_C)x_{{n_k}}\|+\|A\|\|Ax_{{n_k}}-By_{{n_k}}\|)\rightarrow 0, n\rightarrow \infty,
\end{aligned}
\]
\[
\|\widehat{y}_{n_k}-y_{n_k}\|\leq \gamma_n(\|(I-P_Q)y_{{n_k}}\|+\|B\|\|Ax_{{n_k}}-By_{{n_k}}\|)\rightarrow 0, n\rightarrow \infty.
\]
Further, we get
\begin{equation}\label{xx1v}
\|x_{{n_k}+1}-x_{n_k}\|=\delta_{n_k}\|f(x_{n_k})-x_{n_k}\|+(1-\delta_{n_k})\|\widehat{x}_{n_k}-x_{n_k}\|\rightarrow 0, n\rightarrow \infty,
\end{equation}
\begin{equation}\label{xx2v}
\|y_{{n_k}+1}-y_{n_k}\|=\delta_{n_k}\|g(y_{n_k})-y_{n_k}\|+(1-\delta_{n_k})\|\widehat{y}_{n_k}-y_{n_k}\|\rightarrow 0, n\rightarrow \infty.
\end{equation}
Hence, we have $\underset{k\rightarrow \infty}{\limsup}\ \langle f(x^*)-x^*,x_{n_k+1}-x^*\rangle\leq 0$ and  $\underset{k\rightarrow \infty}{\limsup}\ \langle g(y^*)-y^*,y_{n_k+1}-y^*\rangle\leq 0$, which implies $\underset{k\rightarrow \infty}{\limsup}\ \tau_{n_k}\leq 0$. By {Lemma \ref{lem2.5}}, we obtain $\underset{n\rightarrow \infty}{\lim}\theta_n=0$. This implies that $(x_n,y_n)\rightarrow (x^*,y^*)$.\qed
\end{proof}

\subsection{Self-adaptive relaxed viscosity-type CQ algorithm (SRVCQA)}
In this section, we first set the same $C$, $Q$, $C_n$ and $Q_n$ as in Theorem \ref{sep2}. In addition, the combination of the relaxed CQ algorithm and the viscosity-type algorithm has the following algorithm
\begin{equation}\label{VRCQ}
		\left\{\begin{aligned}
&{\widehat{x}_{n}=x_{n}-\gamma_n\left[(I-P_{C_n})x_{n}+A^*(Ax_{n}-By_{n})\right],}\\
&{x_{n+1}=\delta_n f(\widehat{x}_n)+(1-\delta_n)\widehat{x}_{n},}\\
&{\widehat{y}_{n}=y_{n}-\gamma_n\left[(I-P_{Q_n})y_{n}-B^*(Ax_{n}-By_{n})\right],}\\
&{y_{n+1}=\delta_n g(\widehat{y}_n)+(1-\delta_n)\widehat{y}_{n}, n\geq 0.}
\end{aligned}\right.
\end{equation}
\noindent The corresponding parameters satisfy the following restrictions
\begin{description}
    \item[\textbf{(E3)}] If $Ax_{n}\neq By_{n}$, the self-adaptive stepsize
    \[
        \gamma_n=\alpha_n\min\left\{1,\frac{\|Ax_n-By_n\|^2}{\|A^*(Ax_{n}-By_{n})\|^2+\|B^*(Ax_{n}-By_{n})\|^2}\right\}\ \text{with}\ \alpha_n\in(0,1).
        \]
         Otherwise, the self-adaptive stepsize $\gamma_n=\alpha_n$;
    \item[\textbf{(E4)}] $\{\delta_n\}$ is real number sequence in $(0,1)$ such that $\lim_{n\rightarrow \infty}\delta_n=0$ and $\sum_{n=0}^{\infty}\delta_n=\infty$.
\end{description}

\begin{theorem}\label{sep4}
Given the assumptions {(V1)-(V3)} and conditions {(E3)-(E4)}. The iterative sequence $\{(x_n, y_n)\}$ generated by algorithm \eqref{VRCQ} strongly converges to $(x^*,y^*)=P_\Omega(f(x^*),g(y^*))$.
\end{theorem}

\begin{proof}
According to the proof of {Theorem \ref{sep2}} and {Theorem \ref{sep3}}, it follows from {Lemma \ref{lem2.5}} that the sequence $\{(x_n, y_n)\}$ converges strongly to  $(x^*,y^*)$.\qed
\end{proof}

\begin{remark}
Obviously, when the contraction mappings are constant mappings, that is, $f\equiv u$ and $g\equiv v$. This shows that the viscosity algorithm is equivalent to the Halpern algorithm.  It follows that the self-adaptive viscosity-type CQ algorithm (SVCQA) \eqref{VCQ} in {Theorem \ref{sep3}} is equivalent to the following Halpern algorithm
\begin{equation*}
		\left\{\begin{aligned}
&{\widehat{x}_{n}=x_{n}-\gamma_n\left[(I-P_C)x_{n}+A^*(Ax_{n}-By_{n})\right],}\\
&{x_{n+1}=\delta_n u+(1-\delta_n)\widehat{x}_{n},}\\
&{\widehat{y}_{n}=y_{n}-\gamma_n\left[(I-P_Q)y_{n}-B^*(Ax_{n}-By_{n})\right],}\\
&{y_{n+1}=\delta_n v+(1-\delta_n)\widehat{y}_{n}, n\geq 0.}
\end{aligned}\right.
\end{equation*}
This algorithm is actually algorithm \eqref{HCQ} in Corollary \ref{cor1}. Furthermore, the self-adaptive relaxed viscosity-type CQ algorithm  (SRVCQA) \eqref{VRCQ} in Theorem \ref{sep4} is equivalent to the following Halpern algorithm
\begin{equation*}
		\left\{\begin{aligned}
&{\widehat{x}_{n}=x_{n}-\gamma_n\left[(I-P_{C_n})x_{n}+A^*(Ax_{n}-By_{n})\right],}\\
&{\widehat{y}_{n}=y_{n}-\gamma_n\left[(I-P_{Q_n})y_{n}-B^*(Ax_{n}-By_{n})\right],}\\
&{x_{n+1}=\delta_n u+(1-\delta_n)\widehat{x}_{n},}\\
&{y_{n+1}=\delta_n v+(1-\delta_n)\widehat{y}_{n},  n\geq 0.}
\end{aligned}\right.
\end{equation*}
This algorithm is actually algorithm \eqref{RCQ1} in Corollary \ref{cor2}.
\end{remark}

\section{Numerical experiments}\label{se5}
In this section, all codes were written in Matlab R2018b, and ran on a Lenovo ideapad 720S with 1.6 GHz Intel Core i5 processor and 8GB of RAM. We consider some numerical experiments to demonstrate the efficiency of our results and compare them with the existing alternating CQ algorithm (ACQA) in Moudafi \cite{moudafi2014alternating}, simultaneous CQ algorithm (SCQA) in Byrne and Moudafi \cite{byrne2012extensions} and the following Dong et al. algorithm in \cite{dong2015solving}.

\begin{theorem} \emph{\cite{dong2015solving}}
Let $H_1$, $H_2$ and $H_3$ be Hilbert spaces, and $C$, $Q$ be two nonempty closed subsets of $H_1$, $H_2$, respectively. Let $A:H_1\rightarrow H_3$, $B:H_2\rightarrow H_3$ be bounded linear operators, and $A^*$, $B^*$ be the adjoint operators of $A$, $B$, respectively. Let $f$ and $g$ be two contraction mappings with coefficients $\lambda_1\in (0,\frac{\sqrt{2}}{2})$, $\lambda_2\in (0,\frac{\sqrt{2}}{2})$, respectively. The iterative sequence $\{(x_n,y_n)\}$ of the split equality problem \eqref{sep} is generated by the following iterative scheme
\begin{equation}\label{Dong}
\left\{
\begin{aligned}
&{x_{n+1}=\delta_nf(x_n)+(1-\delta_n)P_C(x_n-\gamma_nA^*(Ax_n-By_n)),\quad}\\
&{y_{n+1}=\delta_ng(y_n)+(1-\delta_n)P_Q(y_n+\gamma_nB^*(Ax_n-By_n)),\ \forall n\geq 0,}
\end{aligned}
\right.
\end{equation}
where $\delta_n\in (0,1)$ such that $\lim_{n\rightarrow \infty}\delta_n=0$, $\sum_{n=0}^{\infty}\delta_n=\infty$ and the stepsize \[
\gamma_n=\alpha_n\min \left\{\frac{\|Ax_n-By_n\|^2}{\|A^*(Ax_n-By_n)\|^2},\frac{\|Ax_n-By_n\|^2}{\|B^*(Ax_n-By_n)\|^2}\right\},\ \alpha_n\in(0,1).
\]
Then the iterative sequence $\{(x_n,y_n)\}$ converges in norm to $(x^*,y^*)\in\Omega$.
\end{theorem}

\begin{remark}
From the numerical results of Dong et al. algorithm in \cite{dong2015solving}, they considered an good stepsize selection based on their own algorithm \eqref{Dong}, that is,
\[
\gamma_n=0.65\times \min \left\{\frac{\|Ax_n-By_n\|^2}{\|A^*(Ax_n-By_n)\|^2},\frac{\|Ax_n-By_n\|^2}{\|B^*(Ax_n-By_n)\|^2}\right\}.
\]
\end{remark}

Based on the above results, we will carry out the following work and obtain the corresponding numerical results to characterize the effectiveness and superiority of our algorithms.

\textbf{(Test environment)} According to the setting conditions of the split equality problem \eqref{sep}, we choose the following conditions: $H_1=H_2=H_3=R^3$, $C=\{(x_1,x_2,x_3)\in H_1,\ x_2^2+x_3^2-1\leq 0\}$, and $Q=\{(y_1,y_2,y_3)\in H_2,\ y_1^2-y_2+5\leq 0\}$, in addition, $A=\left[\begin{smallmatrix}{\sqrt{5}} & {0} & {0} \\ {0} & {5} & 0  \\ {0} & {0} & {1}\end{smallmatrix}\right]$ and $ B=\left[\begin{smallmatrix}{1} & {0} & {0} \\ {0} & {1} & {0} \\ {0} & {0} & {1}\end{smallmatrix}\right]$.

By the above matrixes $A$ and $B$, we can easily get the corresponding adjoint operators of $A$ and $B$, that is, $A^*=A^T$ and $B^*=B^T$. Under the above assumption, it is easy to prove that $(x^*,y^*)$ is a unique solution of problem \eqref{sep}, where $x^*=(0,1,0)$, $y^*=(0,5,0)$. The norm $\|Ax_n-By_n\|^2$ as an error estimate and denoted by $E_n$ for all of the following examples. Next, we study and analyze our numerical experiments in such an environment.

\begin{example}\label{exam1}
In the above test environment, we will analyze the convergence of algorithm \eqref{VCQ} (SVCQA) in {Theorem \ref{sep3}}. Firstly, initial points $x_0$, $y_0$ generated randomly in $R^3$,
\[
\gamma_n=\alpha_n\min\left\{1,\frac{\|Ax_n-By_n\|^2}{\|A^*(Ax_{n}-By_{n})\|^2+\|B^*(Ax_{n}-By_{n})\|^2}\right\}\ \text{ with }\ \alpha_n=\frac{3n}{3n+1},
\]
and take contraction mappings $f(x)=0.5 x$, $g(y)=0.5 y$. We consider the following four cases of the parameter $\delta_n$: (a) $ \delta_n=\frac{1}{n+1}$, (b) $\delta_n=\frac{1}{n+10}$, (c) $\delta_n=\frac{1}{n+30}$, (d) $\delta_n=\frac{1}{n+50}$.
The numerical results of algorithm \eqref{VCQ} (SVCQA) for any initial points $x_0$, $y_0$ as shown in {Figure \ref{fig1}}.
\end{example}

\begin{figure}[htbp]
	\centering  
	\subfigure[Case I]{
		\label{Fig1sub1}
		\includegraphics[width=0.45\textwidth]{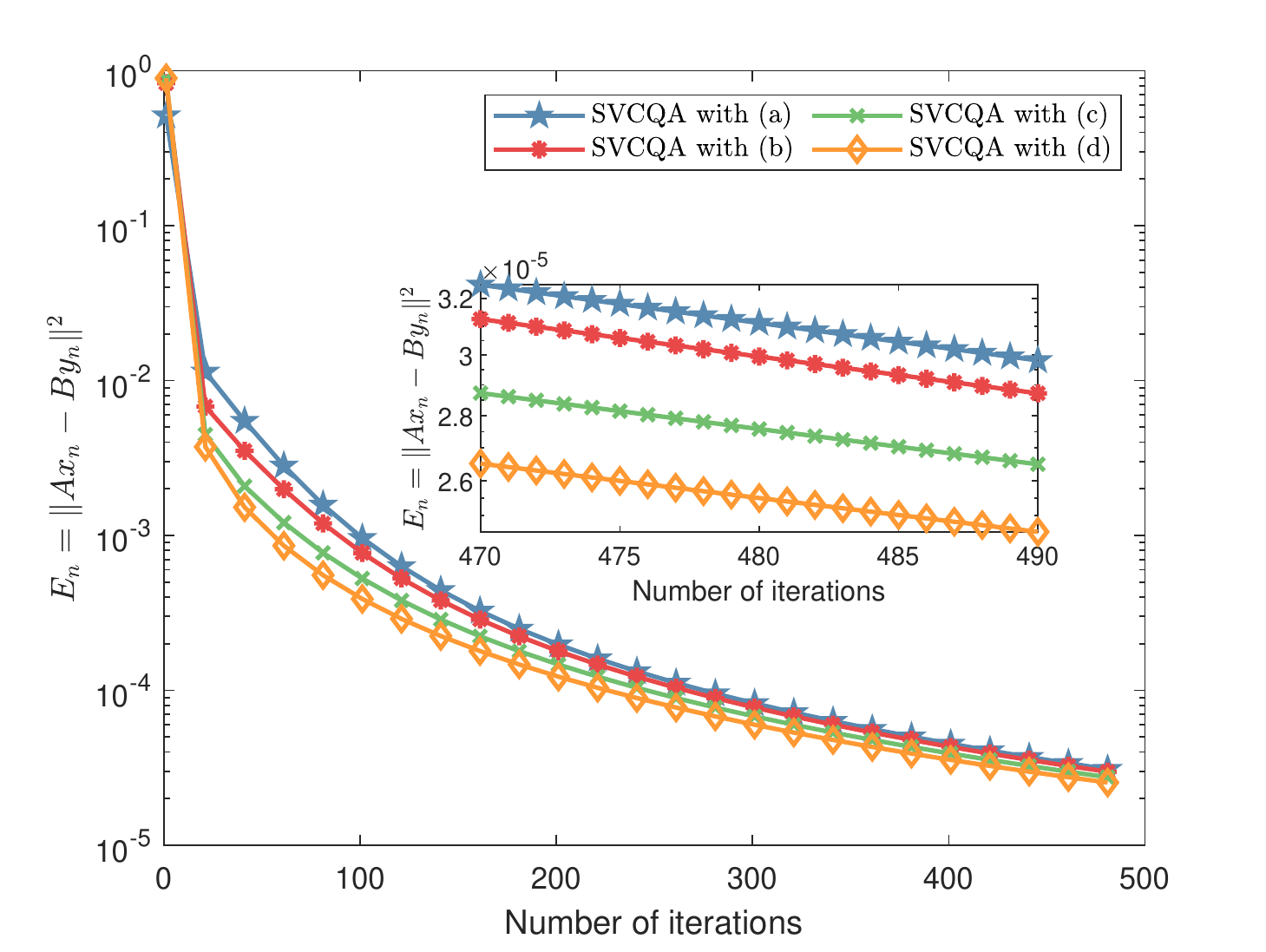}}
	\subfigure[Case II]{
		\label{Fig1sub2}
		\includegraphics[width=0.45\textwidth]{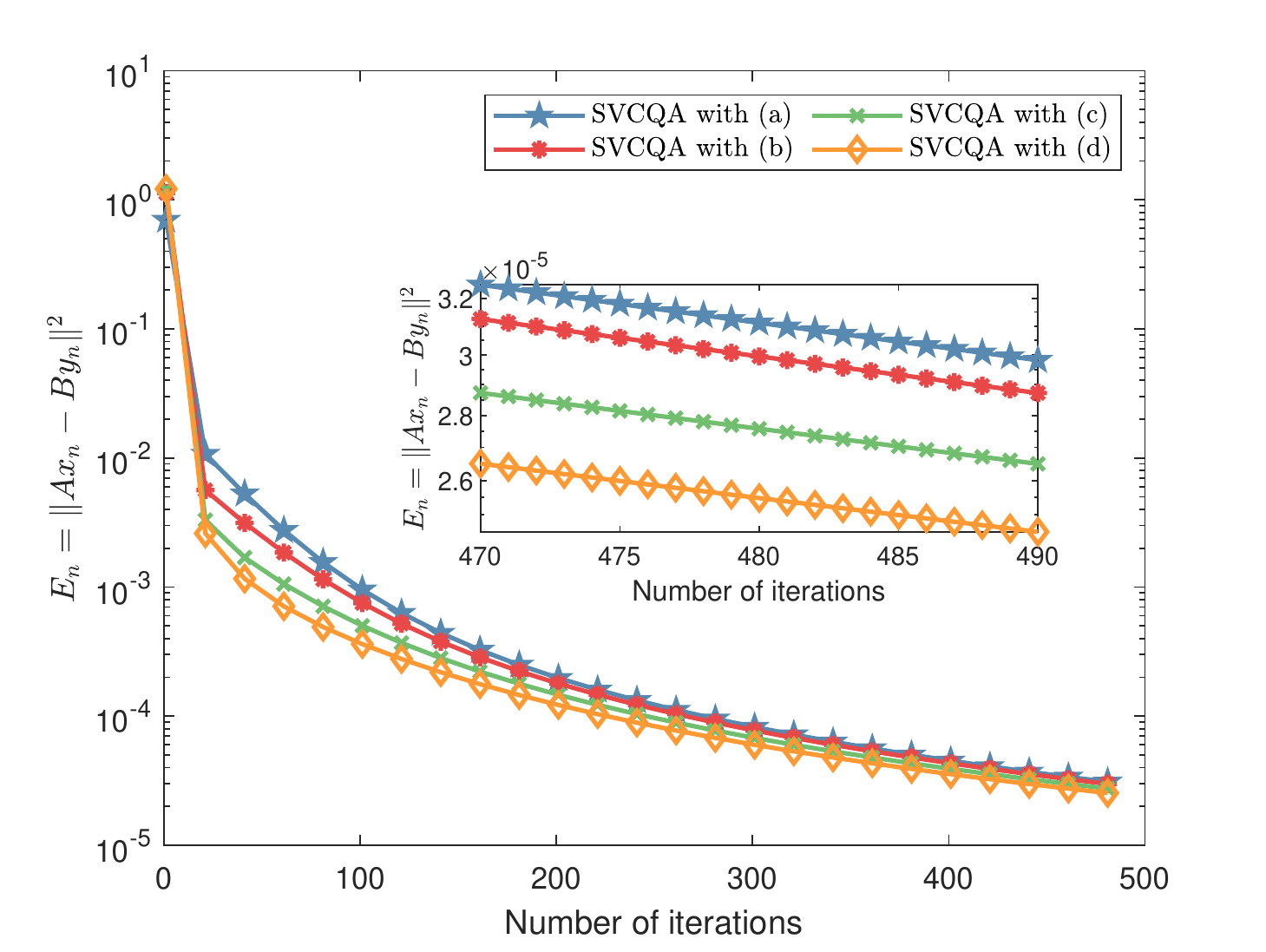}}
	\subfigure[Case III ]{
		\label{Fig1sub3}
		\includegraphics[width=0.45\textwidth]{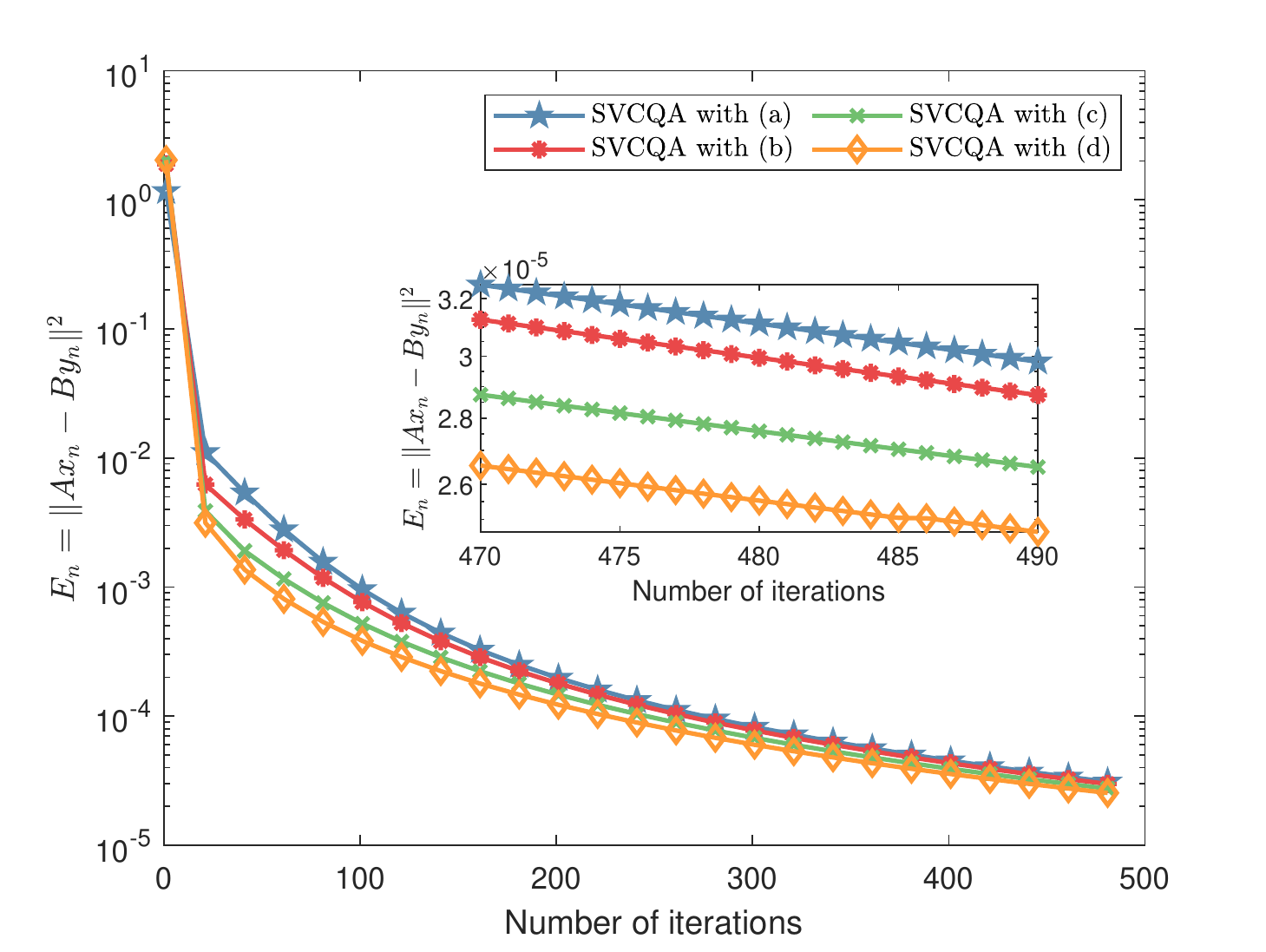}}
	\subfigure[Case IV ]{
		\label{Fig1sub4}
		\includegraphics[width=0.45\textwidth]{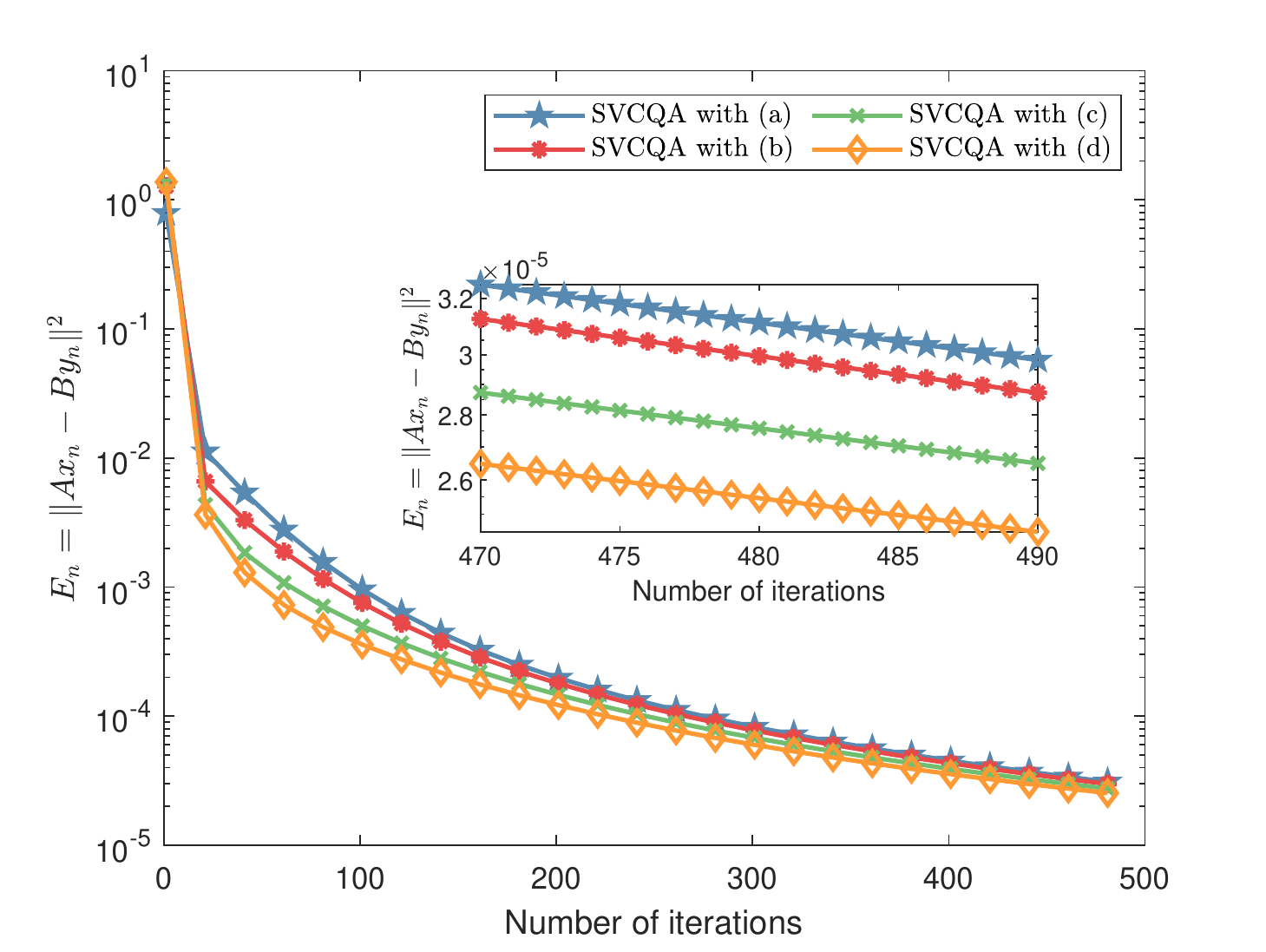}}
	\caption{The numerical results of four parameter choices of SVCQA}
	\label{fig1}
\end{figure}

\begin{remark}
In Figure \ref{fig1}, we can easily see that all the results are valid and convergent. Under the same number of iterations, the error accuracy of the fourth setting $(d)$ is better than all other cases for the different parameters $\delta_n$ in algorithm \eqref{VCQ}. In view of this, we choose $\delta_n=\frac{1}{n+50}$ in algorithm \eqref{VCQ}.
\end{remark}

\begin{example}\label{exam2}
For algorithm \eqref{VCQ} (SVCQA) in {Theorem \ref{sep3}}, we further consider the choice of contraction mappings $f$ and $g$. Firstly, initial points $x_0$, $y_0$ generated randomly in $R^3$,
\[
\gamma_n=\alpha_n\min\left\{1,\frac{\|Ax_n-By_n\|^2}{\|A^*(Ax_{n}-By_{n})\|^2+\|B^*(Ax_{n}-By_{n})\|^2}\right\}\ \text{ with }\ \alpha_n=\frac{3n}{3n+1},
\]
we choose directly the parameter $\delta_n=\frac{1}{n+50}$ and consider different contraction mappings $f(x)=\lambda_1x$ and $g(y)=\lambda_2y$, $\lambda=\lambda_1=\lambda_2\in [0, {1}/{\sqrt{2}})$ for any $x\in H_1$, $y\in H_2$. The numerical results of algorithm \eqref{VCQ} (SVCQA) for any initial points $x_0$, $y_0$ are shown in {Figure \ref{fig2}}.
\end{example}

\begin{figure}[htbp]
	\centering  
	\subfigure[Case I]{
		\label{Fig2sub1}
		\includegraphics[width=0.45\textwidth]{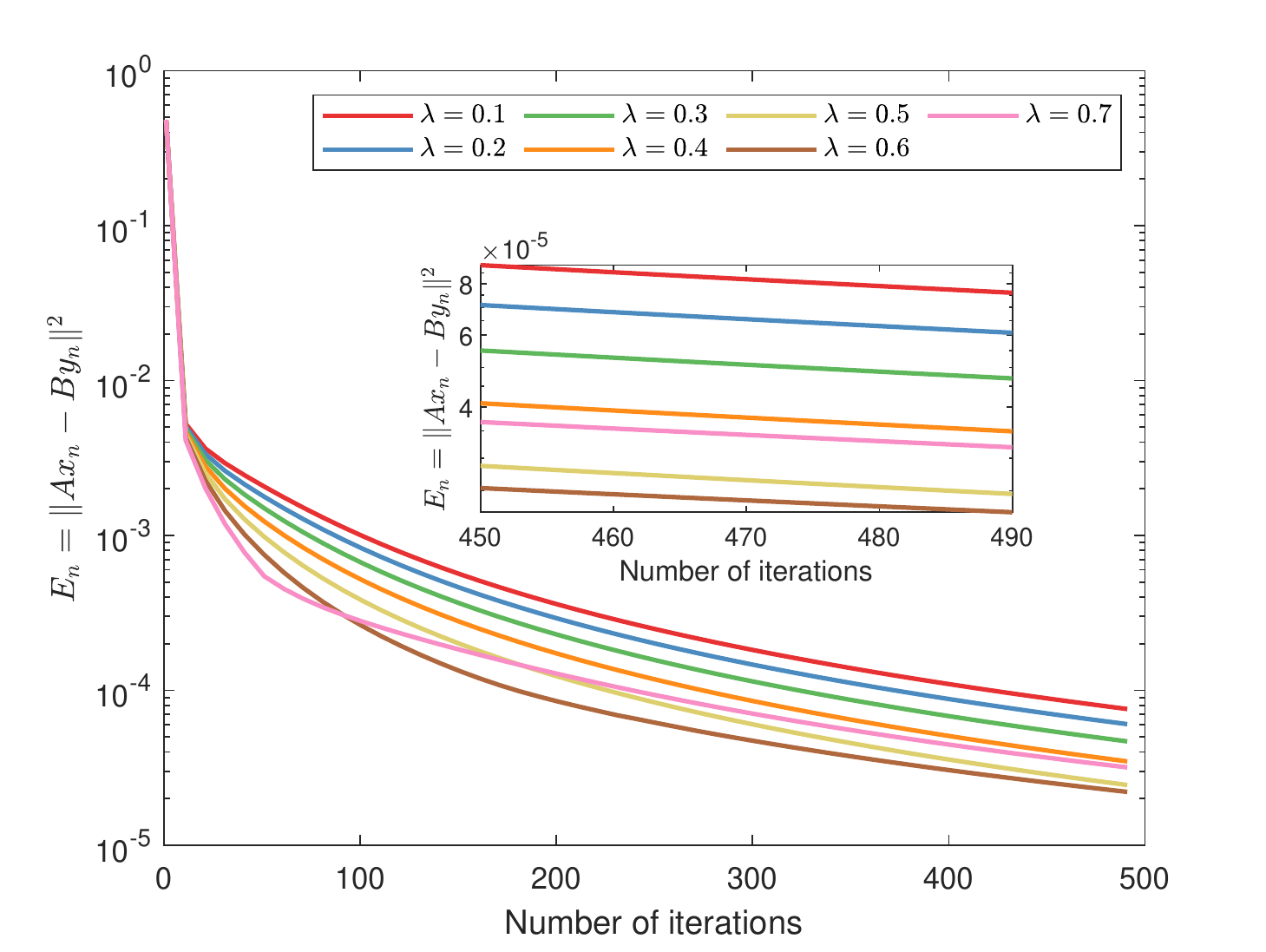}}
	\subfigure[Case II]{
		\label{Fig2sub2}
		\includegraphics[width=0.45\textwidth]{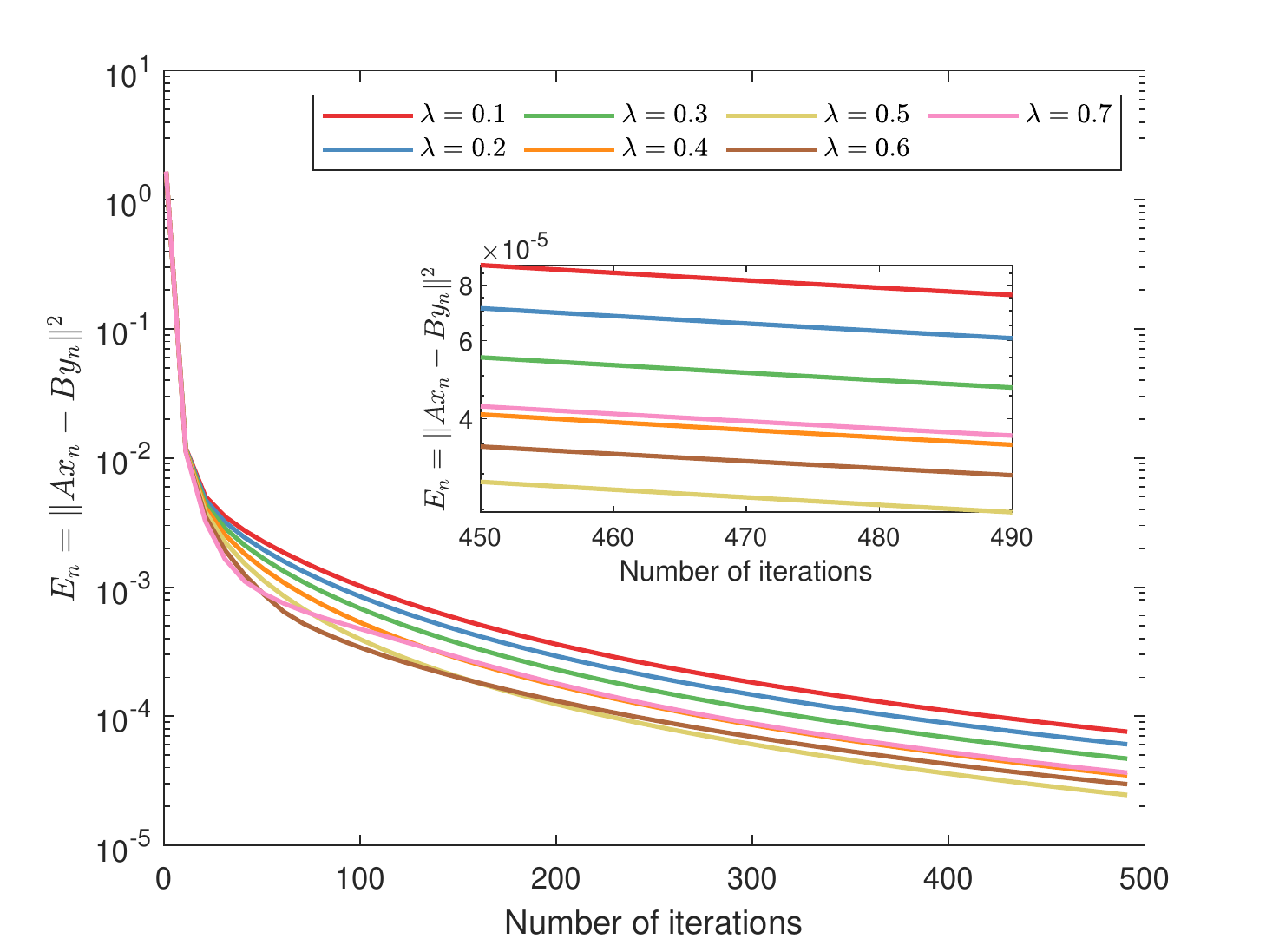}}
	\subfigure[Case III ]{
		\label{Fig2sub3}
		\includegraphics[width=0.45\textwidth]{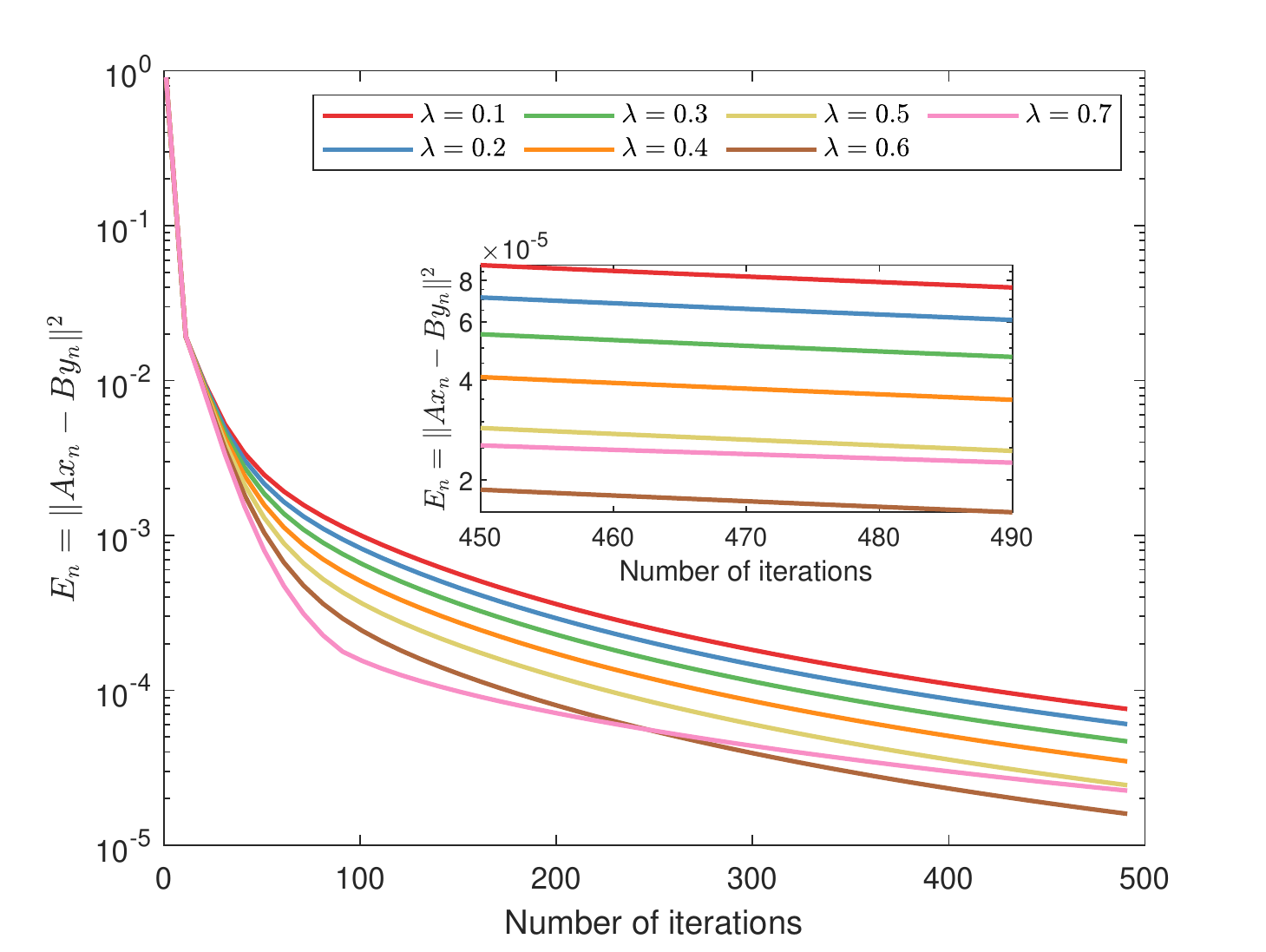}}
	\subfigure[Case IV ]{
		\label{Fig2sub4}
		\includegraphics[width=0.45\textwidth]{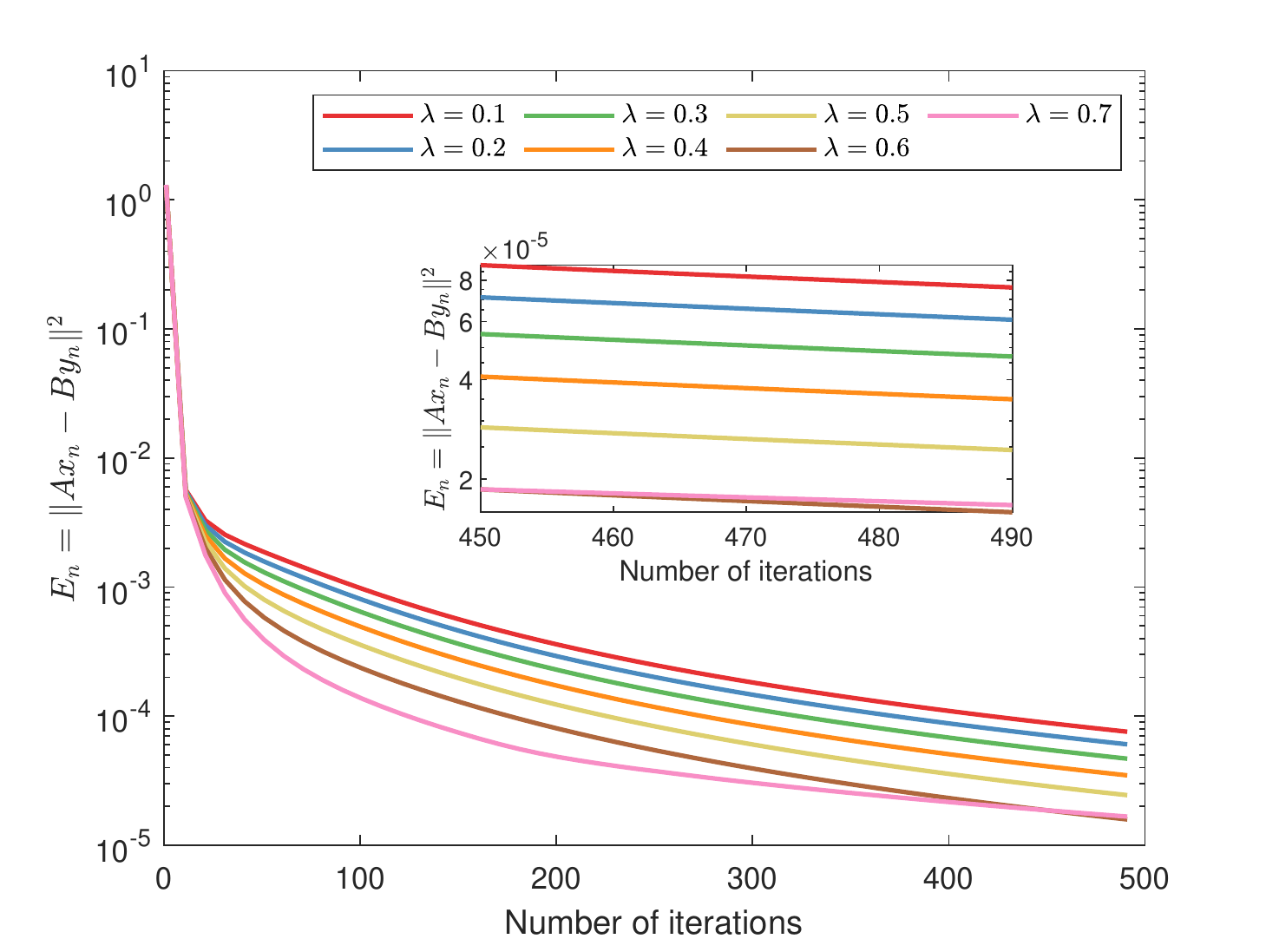}}
\caption{The numerical results of the different contraction mappings of SVCQA}
	\label{fig2}
\end{figure}

\begin{remark}
\begin{itemize}
\item In {Figure \ref{fig2}}, we can see that all the results are valid and convergent. {Figure \ref{fig2}} shows that the coefficients $\lambda_1$ and $\lambda_2$ of contraction mappings $f$ and $g$ have better convergence results in the range $(0.5,0.7)$ for any initial points under the parameter $\delta_n=\frac{1}{n+50}$.
\item By virtue of the numerical results of {Example 5.1} and {Example 5.2}, we have analyzed the different choices of parameters $\delta_n$ and contractions mapping in the algorithm \eqref{VCQ} (SVCQA) in {Theorem \ref{sep3}}. Further, we have got the best results, when
    \[
     \delta_n=\frac{1}{n+50},\ f(x)=0.6x,\ g(y)=0.6y.
    \]
\end{itemize}
\end{remark}

\begin{example}\label{exa3}
For the four algorithms mentioned in this paper: alternating CQ algorithm {(ACQA)} in Moudafi {\cite{moudafi2014alternating}} (i.e., algorithm \eqref{ACQA}), simultaneous CQ algorithm {(SCQA)} in Byrne and Moudafi {\cite{byrne2012extensions}} (i.e., algorithm \eqref{SCQA}), Dong et al. algorithm in \cite{dong2015solving} (i.e., algorithm \eqref{Dong}) and our algorithm \eqref{VCQ} (SVCQA). We compare the number of iterations of four algorithms with different initial points at the same iteration error accuracy. Firstly, we set the corresponding parameters as follows:
\begin{itemize}
\item Take the parameter $\gamma_n$ in ACQA and SCQA as $0.9 \min(\frac{1}{\|A\|^2},\frac{1}{\|B\|^2})$ and $0.9\frac{2}{\|A\|^2+\|B\|^2}$, respectively;
\item Take the parameters in  algorithm \eqref{Dong} as $\delta_n=\frac{1}{n+50}$, $f(x)=0.6 x$, $g(y)=0.6 y$ and
\[
\gamma_n=0.65 \min \left\{\frac{\|Ax_n-By_n\|^2}{\|A^*(Ax_n-By_n)\|^2},\frac{\|Ax_n-By_n\|^2}{\|B^*(Ax_n-By_n)\|^2}\right\}.
\]
\item Take the parameters in  algorithm \eqref{VCQ} (SVCQA) as $\delta_n=\frac{1}{n+50}$, $f(x)=0.6 x$, $g(y)=0.6 y$ and
    \[
    \gamma_n=\alpha_n\min\left\{1,\frac{\|Ax_n-By_n\|^2}{\|A^*(Ax_{n}-By_{n})\|^2+\|B^*(Ax_{n}-By_{n})\|^2}\right\}\ \text{with}\ \alpha_n=\frac{3n}{3n+1}.
    \]
\end{itemize}
At four different initial values, {Figure \ref{fig3}} shows the convergence behavior of the iteration error $E_n$ of the four algorithms under the same number of iterations, and {Table \ref{tab1}} shows the number of iterations of the four algorithms at the same iteration error accuracy.
\end{example}

\begin{figure}[htbp]
	\centering  
	\subfigure[Case I]{
		\label{Fig3sub1}
		\includegraphics[width=0.4\textwidth]{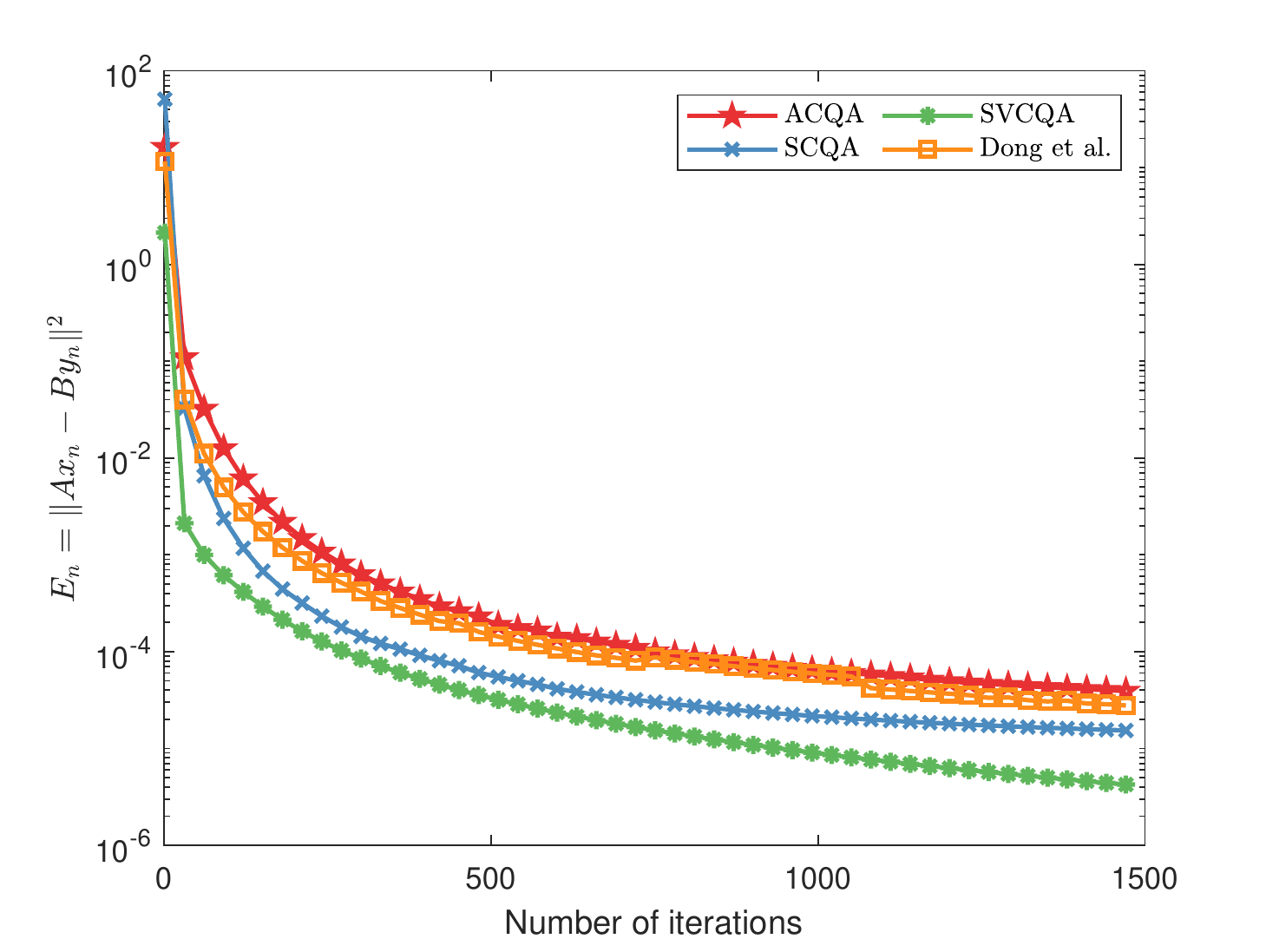}}
	\subfigure[Case II]{
		\label{Fig3sub2}
		\includegraphics[width=0.4\textwidth]{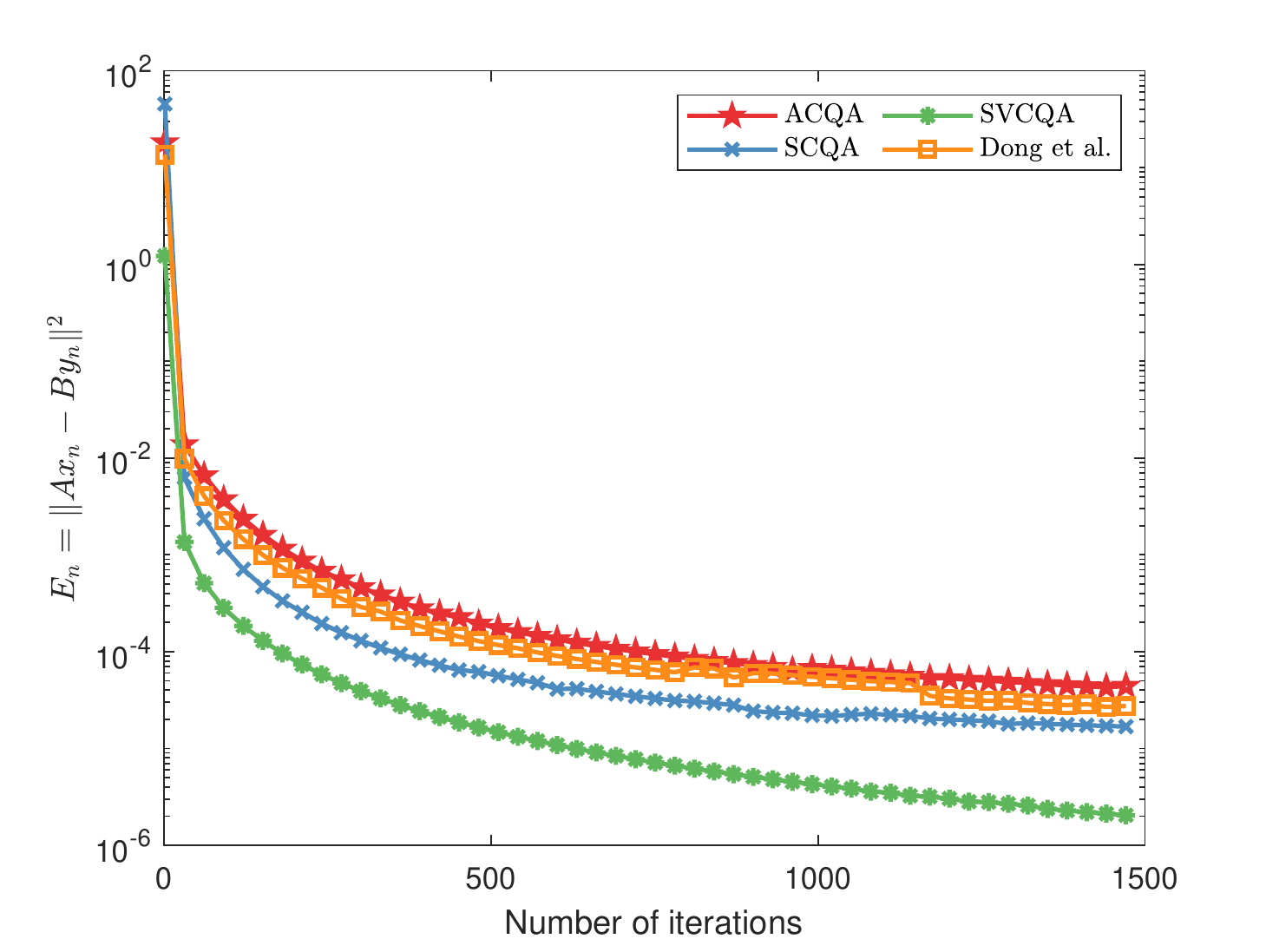}}
	\subfigure[Case III ]{
		\label{Fig3sub3}
		\includegraphics[width=0.4\textwidth]{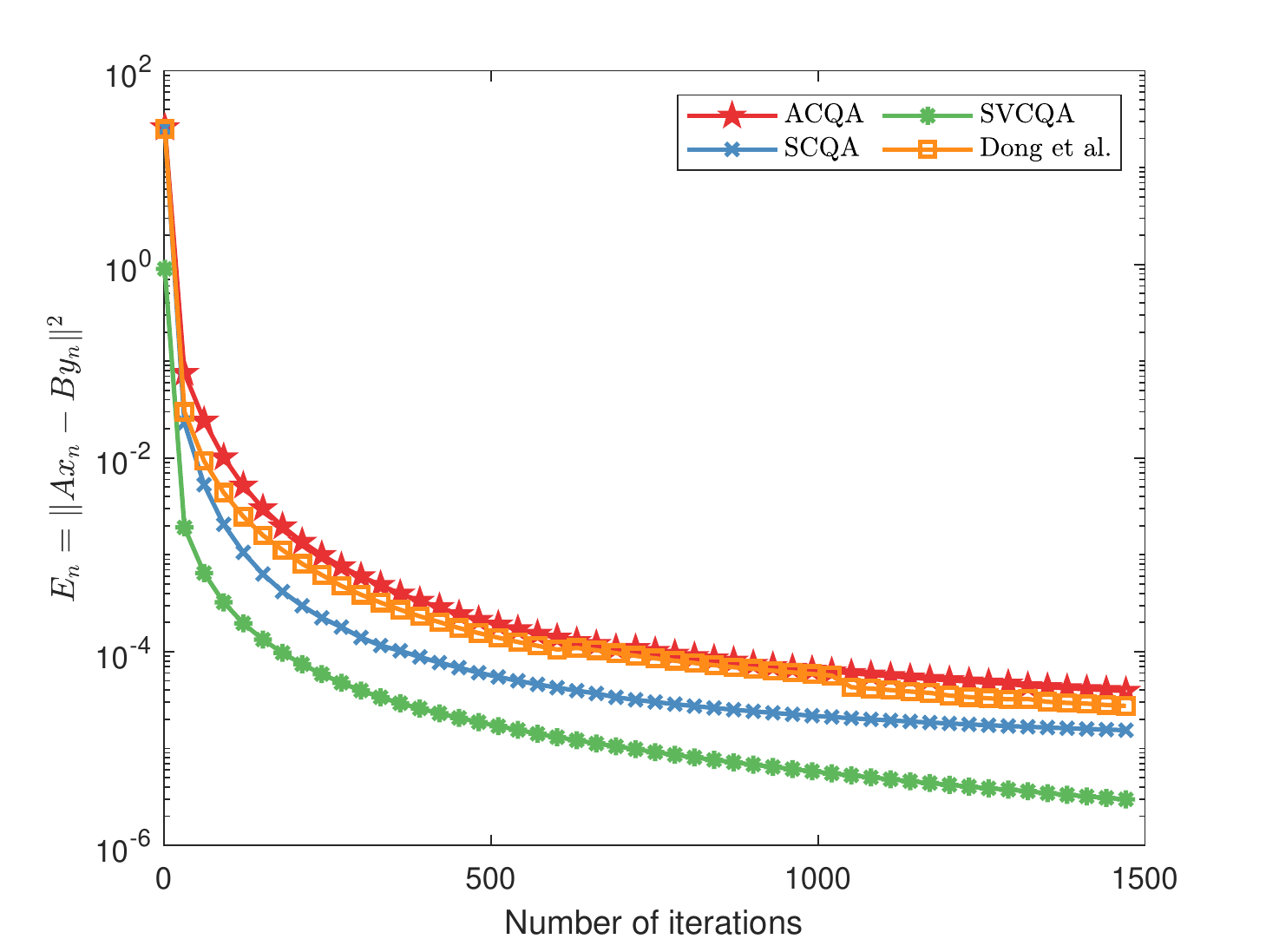}}
	\subfigure[Case IV ]{
		\label{Fig3sub4}
		\includegraphics[width=0.4\textwidth]{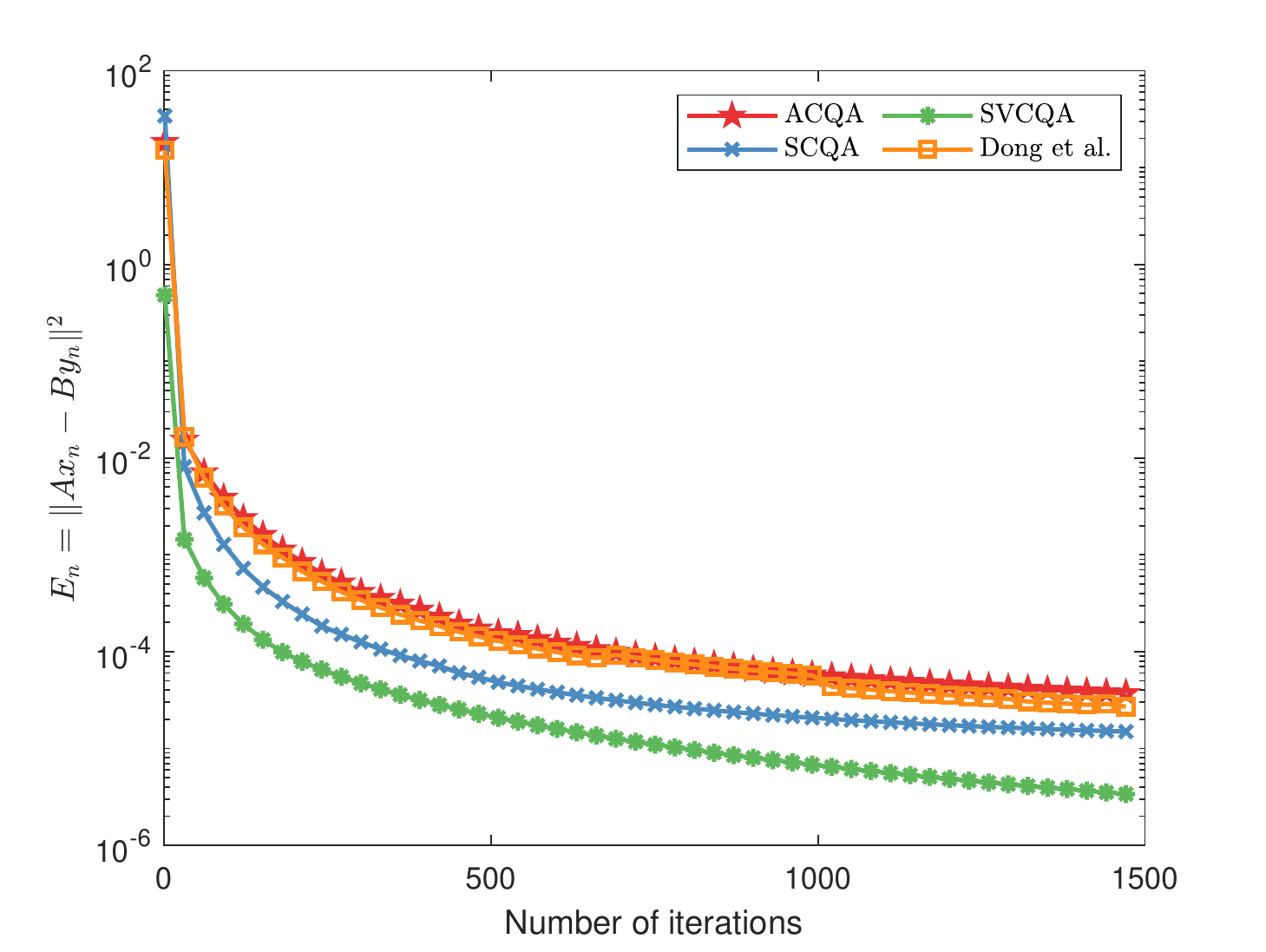}}
	\caption{The numerical results of four algorithms}\label{fig3}
\end{figure}

\begin{table}[tp]
\newcommand{\tabincell}[2]{\begin{tabular}{@{}#1@{}}#2\end{tabular}}
  \centering
  \caption{Number of iterations for different error estimates}
  \label{tab1}
    \begin{tabular}{cccccc}
    \toprule
    \multirow{2}{*}{initial point $(x_0,y_0)$}&
    \multirow{2}{*}{$\|Ax_n-By_n\|^2$}&
    \multicolumn{3}{c}{ Number of iterations}\\
    \cmidrule(lr){3-6}
    &&SVCQA&ACQA&SCQA&Dong et al.\\
    \midrule
    \multirow{3}{*}{ \tabincell{c}{$x_0=(0.7922;0.9595;0.6557)^\mathrm{T}$\\
    $y_0=(0.0357;0.8491;0.9340)^\mathrm{T}$}}&$10^{-2}$&11&100&52&65\\
    &$10^{-3}$&61&249&129&197\\
    &$10^{-4}$&276&750&373&627\\
     \midrule
    \multirow{3}{*}{ \tabincell{c}{$x_0=(0.6787;0.7577;0.7431)^\mathrm{T}$\\
    $y_0=(0.3922;0.6555;0.1712)^\mathrm{T}$}}&$10^{-2}$&8&43&21&31\\
    &$10^{-3}$&39&196&100&151\\
    &$10^{-4}$&177&721&349&564\\
     \midrule
    \multirow{3}{*}{ \tabincell{c}{$x_0=(0.7060;0.0318;0.2769)^\mathrm{T}$\\
    $y_0=(0.0462;0.0971;0.8235)^\mathrm{T}$}}&$10^{-2}$&10&92&47&59\\
    &$10^{-3}$&48&241&124&192\\
    &$10^{-4}$&178&757&366&624\\
     \midrule
    \multirow{3}{*}{ \tabincell{c}{$x_0=(0.1190;0.4984;0.9597)^\mathrm{T}$\\
    $y_0=(0.3404;0.5853;0.2238)^\mathrm{T}$}}&$10^{-2}$&7&47&27&45\\
    &$10^{-3}$&42&193&104&173\\
    &$10^{-4}$&180&686&343&596\\
    \bottomrule
    \end{tabular}
\end{table}

\begin{remark}
From Figure \ref{fig3} and Table \ref{tab1}, we can see that our proposed algorithm \eqref{VCQ} (SVCQA) outperforms the alternating CQ algorithm {(ACQA)}, simultaneous CQ algorithm {(SCQA)}, Dong et al. algorithm \eqref{Dong} in both  error accuracy and number of iterations.
\end{remark}

\section{Conclusion}

The first conclusion from Sections \ref{se3} and \ref{se4} is that we propose four self-adaptive CQ algorithms by using the methods of the Halpern algorithm and the viscosity algorithm for solving the SEP under the condition of the self-adaptive stepsize sequence. A point should be stressed is that such a self-adaptive stepsize sequence does not depend on the prior knowledge of operator norms. The second conclusion from the numerical results in Section \ref{se5} is that the convergence of our algorithm is validity and authenticity. Meanwhile, our proposed self-adaptive viscosity-type CQ algorithm \eqref{VCQ} (SVCQA) improves and extends the existing results.

\end{document}